\documentclass[12pt,a4paper]{amsart}

\usepackage{latexsym} 
 
\usepackage[dvips]{graphics}
\usepackage{epsfig}

\usepackage{amssymb}
\usepackage{amsthm}
\usepackage{amsfonts}
\usepackage{amsmath}
\usepackage{amstext}
\usepackage{amscd}
\usepackage{enumerate}

\numberwithin{equation}{section}

\theoremstyle{plain}
\newtheorem{thm}{Theorem}[section] 
\newtheorem{prop}[thm]{Proposition}
\newtheorem{cor}[thm]{Corollary}
\newtheorem{lem}[thm]{Lemma}

\newtheorem{theorem*}{Theorem}[]

\theoremstyle{definition}
\newtheorem{defn}[thm]{Definition}

\newtheorem{example}[thm]{Example}

\theoremstyle{remark}
\newtheorem{rem}[thm]{Remark}

\newcommand{\C}{\mathbb{C}}
\newcommand{\N}{\mathbb{N}}
\newcommand{\R}{\mathbb{R}}

\newcommand{\Q}{\mathbb{Q}}
\newcommand{\Z}{\mathbb{Z}}

\newcommand{\inv}{^{-1}}

\newcommand{\half}{\frac{1} 2}

\DeclareMathOperator{\id}{id\,}
\DeclareMathOperator{\supp}{supp\,}
\DeclareMathOperator{\mult}{mult\,}
\DeclareMathOperator{\ord}{ord\,}

\DeclareMathOperator{\grad}{grad}

\DeclareMathOperator{\NB}{NB}

\def\accentclass@{7}
\def\makeacc@#1#2{\def#1{\mathaccent"\accentclass@#2 }}
\makeacc@\cir{017}

\title[Equivalence relations for real analytic function germs]
{Equivalence relations for two variable real analytic function germs}

\author{Satoshi Koike \& Adam Parusi\'nski }

\address {Department of Mathematics, Hyogo University
of Teacher Education, 942-1 Shimokume, Kato,
Hyogo 673-1494, Japan}

\email {koike@hyogo-u.ac.jp}

\address {Laboratoire Angevin de Recherche en Math\'ematiques, UMR
  6093 du CNRS, Universit\'e d'Angers,
   2, bd Lavoisier, 49045 Angers cedex, France}

\email{adam.parusinski@univ-angers.fr}

\subjclass{Primary: 32S15. Secondary: 14B05, 57R45}

\newcommand{\abstracttext}{For two variable real analytic function germs we compare 
the blow-analytic equivalence in the sense of Kuo to the other natural equivalence relations.  
Our  main theorem states that $C^1$ equivalent germs are  blow-analytically 
equivalent. This gives a negative answer to a conjecture of Kuo.  In the proof we show that 
the Puiseux pairs of real Newton-Puiseux roots are preserved by the $C^1$ 
equivalence of function germs.  
The proof is achieved, being based on a combinatorial characterisation of 
blow-analytic equivalence in terms 
of the real tree model. 

We also give several examples of bi-Lipschitz equivalent germs that are not blow-analytically equivalent.  }

\begin{document}

\begin{abstract} \abstracttext\end{abstract}

\keywords{Blow-analytic equivalence, Tree model, Puiseux pairs, $C^1$ equivalence, bi-Lipschitz equivalence}

\footnotetext[1]{The second named author was partially supported by the JSPS  Invitation Fellowship Program. ID No. S-07026}

\maketitle


\vspace{1cm}  




The natural equivalence relations we first think of
are the $C^r$ coordinate changes for $r = 1, 2, \cdots , \infty ,
\omega$, where $C^\omega$ stands for real analytic.  
Let $f$, $g : (\R^n,0) \to (\R,0)$ be real analytic function germs.
We say that $f$ and $g$ are $C^r$ {\em (right)  equivalent}
if there is a local $C^r$ diffeomorphism
$\sigma : (\R^n,0) \to (\R^n,0)$ such that
$f = g \circ \sigma$. 
If $\sigma$ is a local bi-Lipschitz homeomorphism, 
resp. a local homeomorphism, then  we say that 
$f$ and $g$ are {\em bi-Lipschitz equivalent}, resp. $C^0$ {\em equivalent}.
By definition, we have the following implications:
\begin{equation}\label{implications} 
\text {$C^0$-eq. $\Leftarrow$ bi-Lipschitz eq. 
$\Leftarrow C^1$-eq. $\Leftarrow C^2$-eq. $\Leftarrow \cdots 
\Leftarrow C^{\infty}$-eq. $ \Leftarrow C^{\omega}$-eq.}
\end{equation}
By Artin's Approximation Theorem \cite{artin},
$C^{\infty}$ equivalence implies 
$C^{\omega}$ equivalence.
But the other converse implications of \eqref{implications} do not hold.
Let $f$, $g : (\R^2,0) \to (\R,0)$ be polynomial functions
defined by
$$
f(x,y) = (x^2 + y^2)^2, \ \ g(x,y) = (x^2 + y^2)^2 + x^{r+4}
$$
for $r = 1, 2, \cdots$.
N. Kuiper \cite{kuiper} and F. Takens \cite{takens}
showed that $f$ and $g$ are $C^r$ equivalent,
but not $C^{r+1}$ equivalent. 

In the family of germs 
$$
K_t(x,y) = x^4 + t x^2 y^2 + y^4,
$$
the phenomenon of continuous $C^1$ moduli appears: for $t_1$, $t_2 \in I$, $K_{t_1}$ 
and $K_{t_2}$ are $C^1$ equivalent
if and only if $t_1 = t_2$, where $I = (- \infty , -6]$,
$[-6,-2]$ or $[-2,\infty )$, see example \ref{c1moduli} below.
On the other hand, T.-C. Kuo proved that this family is 
$C^0$-trivial over  any interval not
containing $-2$, by a $C^0$ trivialisation obtained by the integration of a vector field, c.f.  \cite{kuo1}.
In the homogeneous case, as that of $K_t$, the Kuo vector field is Lipschitz and the trivialisation is bi-Lipschitz.  
Thus Kuo's construction gives examples of bi-Lipschitz equivalent germs that are not $C^1$ equivalent.  

It is easy to construct examples of $C^0$ equivalent and bi-Lipschitz non-equivalent germs.  
Let us  note that, moreover, the bi-Lipschitz equivalence also has continuous moduli, c.f. 
 \cite{henryparusinski1, henryparusinski2}.  For instance the family 
$$
A_t(x,y) = x^3 - 3 t x y^4 + 2 y^6, \quad t>0 , 
$$
is $C^0$ trivial and  if $A_{t_1}$ is bi-Lipschitz equivalent to $A_{t_2}$, $t_1,t_2>0$, then $t_1=t_2$.


\subsection {Blow-analytic equivalence} 
Blow-analytic equivalence was proposed for 
real analytic function germs by Tzee-Char Kuo 
\cite{kuo3} as a counterpart of the topological equivalence 
of complex analytic germs.  Kuo showed in \cite{kuo8} the local finiteness (i.e. the 
absence of continuous moduli) 
of blow-analytic types for analytic families of  isolated singularities.

We say that a homeomorphism germ $h :
(\R^n,0) \to (\R^n,0)$
is  a {\em blow-analytic homeomorphism}  if there exist real modifications 
$\mu : (M,\mu^{-1}(0)) \to (\R^n,0)$, 
$\tilde \mu : (\tilde M, \tilde \mu\inv (0) )$ $ \to (\R^n,0)$
and an analytic isomorphism $\Phi : (M,\mu^{-1}(0)) \to
(\tilde M,\tilde\mu^{-1}(0))$
so that $\sigma \circ \mu =\tilde \mu \circ \Phi$.  The formal definition of real modification is somewhat technical 
and, since in this paper we consider only the two variable case, we shall use the following criterion of 
\cite {koikeparusinski2}:  in two variable case $\mu$ is a real modification if and only if 
it is a finite composition of point blowings-up.  
Finally, we say that two real analytic function germs
$f : (\R^n,0) \to (\R,0)$ and $g : (\R^n,0) \to (\R,0)$ 
are {\em blow-analytically equivalent} 
if there exists a blow-analytic homeomorphism  $\sigma : (\R^n,0) \to (\R^n,0)$
such that $f = g \circ \sigma$.  

 For instance, the family $K_t$, $t\ne -2$, becomes real analytically trivial after the blowing-up 
of the $t-$axis, cf. Kuo \cite {kuo3}.  Thus for $t<-2$, or $t>-2$ respectively, all $K_t$ are blow-analytically 
equivalent.   
Similarly, the family $A_t$ becomes real analytically trivial after a toric blowing-up in $x,y-$variables, cf. 
Fukui - Yoshinaga \cite{fukuiyoshinaga} or 
Fukui - Paunescu \cite{fukuipaunescu1}, 
and hence it is blow-analytically trivial.  Thus blow-analytic equivalence does not imply 
neither $C^r$-equivalence, 
$r \ge 1$, nor   bi-Lipschitz equivalence.

 Blow-analytic equivalence is a stronger and more natural notion than 
$C^0$ equivalence.  For instance, $f(x,y) = x^2 - y^3$, $g(x,y) = x^2 - y^5$ 
are $C^0$ equivalent, but not blow-analytically equivalent.  The latter fact 
 can be seen using the Fukui invariant \cite{fukui}, that we recall in section \ref{bilipschitz} below, 
 or it follows directly from the following theorem.

\begin{thm}\label{allequivalent}{\rm (cf. \cite{koikeparusinski2}) }
Let $f:(\R^2,0)\to (\R,0)$ and $g:(\R^2,0)\to (\R,0)$ be real analytic 
function germs.  Then the following conditions are equivalent:
\begin{enumerate}
\item
$f$ and $g$ are blow-analytically equivalent.
\item 
$f$ and $g$ have isomorphic minimal resolutions.
\item 
The real tree models of $f$ and $g$ are isomorphic.  
\end{enumerate}
\end{thm}

For more on the blow-analytic equivalence in the general case $n$-dimensional 
we refer the reader to recent surveys 
 \cite{fukuikoikekuo, fukuipaunescu2}.

What is the relation between blow-analytic equivalence and $C^r$
equivalences, $1 \le r < \infty$?  
Kuo states in \cite{kuo3} that his modified analytic homeomorphism
is independent of $C^r$ diffeomorphisms, $1 \le r < \infty$, and
confirms 
his belief at the invited address of the annual convention of the
Mathematical Society of Japan, autumn 1984 (\cite{kuo7}),
by asserting that
blow-analytic equivalence is independent of $C^r$ equivalences.
Untill now it was widely believed that this is the case.


\subsection {Main results of this paper}  
The main result of this paper is the following.  

\begin{thm}\label{c1givesblow}
Let $f:(\R^2,0)\to (\R,0)$ and $g:(\R^2,0)\to (\R,0)$ be real analytic 
function germs and suppose that there exists a $C^1$ diffeomorphism germ  $\sigma : (\R^2,0)\to (\R^2,0)$ 
such that $f = g\circ \sigma$.  Then $f$ and $g$ are blow-analytically equivalent.  

If, moreover, $\sigma$ preserves orientation, then $f$ and $g$ are 
blow-analytically equivalent by an orientation preserving blow-analytic homeomorphism.  
\end{thm}

This gives a negative answer to the above conjecture of Kuo.
To give the reader some flavour of this unexpected result we propose the following special 
case that contains most of the difficulty of the proof and does not refer to blow-analytic equivalence.  
Recall that a Newton-Puiseux root of $f(x,y)=0$ is a real analytic arc $\gamma \subset f\inv (0)$ 
parameterised by 
\begin{equation*}
\gamma : x = \lambda (y) = a_1 y^{n_1 / N} + a_2 y^{n_2 / N}
+ \cdots , 
\end{equation*}
where $\lambda(y)$ is a convergent fractional power series $\lambda (y) = a_1 y^{n_1 / N} + a_2 y^{n_2 / N}
+ \cdots $.  We shall always assume that $n_1/N \ge 1$, that is $\gamma$ is transverse to the 
$x$-axis.  We shall call $\gamma$ real if all $a_i$ are real for $y\ge 0$, and then we understand 
$\gamma$ as such real demi-branch of an analytic arc, with the parametrisation restricted to $y\ge 0$.   

\begin{prop} {\rm  (cf. proposition \ref{Puiseuxinvariance1} below)}\\
Let $\sigma : (\R^2,0) \to (\R^2,0)$ be a $C^1$ diffeomorphism 
and let $f$, $g : (\R^2,0) \to (\R,0)$ be real analytic function germs
such that $f = g \circ \sigma$. Suppose that $\gamma \subset f\inv (0)$ and  $\tilde{\gamma}\subset g\inv (0)$
be Newton-Puiseux roots of $f$ and $g$ respectively  such that  $\sigma (\gamma) = \tilde \gamma$ as set germs.  
Then the Puiseux characteristic pairs of $\gamma$ and 
$\tilde \gamma$ coincide.
\end{prop}

This property has no obvious counterpart in the complex set-up.  The Puiseux pairs of plane curve 
singularities are embedded topological invariants, cf. \cite
{zariski1}.  
 We can not dream of any 
similar statement in the real analytic set-up, all real analytic demi-branches are $C^1$ equivalent 
to the positive $y-$axis.  In the proof of proposition  \ref{Puiseuxinvariance1} we use two basic assumptions,  
the arcs are  roots and $\sigma$ conjugates the analytic functions defining the roots: $f = g \circ \sigma$.  

There is another major difference to the complex case.  
The topological type of a complex analytic function germ 
can be combinatorially characterised in terms of the tree model of  \cite{kuolu}, that encodes the 
contact orders between different Newton-Puiseux roots, that give, in particular, the Puiseux pairs of those roots.  
This is no longer true in real words,  the Puiseux pairs cannot be read from these contact orders, see 
\cite{koikeparusinski2}.  

It was speculated for a long time that there is a relation between blow-analytic and bi-Lipschitz 
properties.  It is not difficult to construct examples showing that 
\begin{equation*}
\text {  blow-analytic-eq. $\not \Rightarrow$ bi-Lipschitz eq ,}
\end{equation*}
 as the example $A_t$ above.  
In this paper we construct several examples showing that  
\begin{equation*}
\text {  blow-analytic-eq. $\not \Leftarrow$ bi-Lipschitz eq .}
\end{equation*}
Thus,  there is no direct relation between these two notions.  
Nevertheless, as shown in  \cite{koikeparusinski2},  a blow-analytic homeomorphism that gives 
blow-analytic equivalence between two  2-variable real analytic function germs, preserves the order of contact 
between {\bf non}-parameterised real analytic arcs.   Note that by the curve selection lemma, a subanalytic 
homeomorphism is bi-Lipschitz if and only if it preserves the order of contact between parameterised real 
analytic arcs.   

For more than two variables we have another phenomenon.  
Let $f_t : (\R^3,0) \to  (\R,0)$, $t \in \R$, be the Brianc\c{o}n-Speder 
family
defined by $f_t(x,y,z) = z^5 + tzy^6 + y^7x + x^{15}$.
Although $f_0$ and $f_{-1}$ are blow-analytically equivalent,
any blow-analytic homeomorphism that gives the blow-analytic
equivalence between them does not preserve the order of contact
between some analytic arcs contained in $f_0^{-1}(0)$,  cf. \cite{koike}.


\subsection {Organisation of  this paper}  
In section \ref{invariantbi-L} we construct new invariants of
bi-Lipschitz and
 $C^1$ equivalences. 
These invariants can be nicely described in terms of the Newton polygon relative to a curve, the notion 
introduced in \cite{kuoparusinski1},
 though the reader can follow an alternative 
way  that uses equivalent notions:  the order function and associated polynomials.   
Shortly speaking,  if $f=g\circ \sigma$  with 
$\sigma$ bi-Lipschitz then the Newton boundaries of $f$ relative to an arc $\gamma$ coincides 
with the Newton boundary of $g$ relative to $\sigma(\gamma)$.   If $\sigma$ is $C^1$ and $D\sigma(0)=Id$ 
then, moreover, the corresponding coefficients on the Newton boundaries are identical.  As a direct corollary 
we get the $C^1$ invariance of Puiseux pairs of the Newton-Puiseux roots, see  proposition  \ref{Puiseuxinvariance1}.

In section \ref{$C^1$} we show theorem \ref{c1givesblow}.  The proof is based on theorem \ref{allequivalent} 
so we recall in this section the construction of real tree model.  

In section \ref{arbitraryc1} we extend the  construction of section \ref{invariantbi-L}  to all $C^1$ diffeomorphisms 
(we drop the assumption $D\sigma(0)=Id$).    As a corollary we give a complete classification of 
$C^1$ equivalent weighted homogeneous germs of two variables.  

Section \ref{bilipschitz} contains the construction of examples of  bi-Lipschitz equivalent and blow-analytically 
non-equivalent germs.  This is not simple since such a bi-Lipschitz equivalence cannot be natural.  
Let us first recall  the construction of invariants of bi-Lipschitz 
equivalence of \cite{henryparusinski1, henryparusinski2}.  Suppose that the generic polar curve of $f(x,y)$ has at least two branches $\gamma_i$.  Fix reasonable 
parametrisations of these branches, either by a coordonate as $x=\lambda_i(y)$ or by the distance to 
the origin, and expand $f$ along each such branch.  Suppose that 
the expansions along different branches $f(\lambda_i(y),y) 
= a_i y^s + \cdots$ 
have the same leading exponent $s$, and that the term $y^s$ is sufficiently big  
in comparison to the distance between the branches.  Then the ratio of the 
leading coefficients $a_i/a_j$ is a bi-Lipschitz invariant (and a continuous modulus).  
Our construction of bi-Lipschitz homeomorphism goes along these lines but in the opposite direction.  
First we choose carefully $f(x,y)$, $g(x,y)$  so that such the expansions of $f$, resp. $g$, along polar branches 
are compatible,  so that 
we write down explicitly bi-Lipschitz equivalences between horn neighbourhoods of polar curves of $f$ and $g$,
respectively.  
Then we show that, in our examples, these equivalences can be glued together using partition of unity.   


\subsection {Observations}  
We shall use freely the following widely known facts.  
In the general $n$-variable case, 
the multiplicity of an analytic function germ is a bi-Lipschitz
invariant.  For the $C^1$ equivalence  the initial homogeneous
form, up to linear equivalence, is an invariant.
Indeed, for real analytic functions not identically zero, we have

\begin {lem}\label{initialform}
Let $f$, $g : (\R^n,0) \to (\R,0)$ be analytic function germs
of the form

\vspace{3mm}

\centerline{$f(x) = f_m(x) + f_{m+1}(x) + \cdots ,
\ \ f_m \not\equiv 0$,}

\vspace{2mm}

\centerline{$g(x) = g_k(x) + g_{k+1}(x) + \cdots ,
\ \ g_k \not\equiv 0$.}

\vspace{3mm}

\noindent Suppose that $f$ and $g$ are $C^1$-equivalent.
Then $k = m$ and $f_m$ and $g_m$ are linearly equivalent.

In particular, if homogeneous polynomial functions are 
$C^1$-equivalent, then they are linearly equivalent.
\end{lem}

\begin {proof}
Since $C^1$-equivalence is a bi-Lipschitz equivalence,
$m = k$.
Let $\sigma (x) = (\sigma_1(x), \cdots , \sigma_n(x))$
be a local $C^1$ diffeomorphism such that $f = g \circ \sigma$.
Let us write
$$
\sigma_i (x) = \sum_{j=1}^n a_{ij} x_j + h_i(x)
$$
where $j^1 h_i(0) = 0$, $i = 1, \cdots , n$.
Set $A(x) = (\sum_{j=1}^n a_{1j} x_j, \cdots , 
\sum_{j=1}^n a_{nj} x_j)$.
Since $\sigma$ is a local $C^1$ diffeomorphism, $A$ is a
linear transformation of $\R^n$.
Now we have
$$
f_m(x) + f_{m+1}(x) + \cdots = f(x) = g(\sigma (x))
= g_m(A(x)) + G(x)
$$
where $f_{m+1}(x) + \cdots = o(|x|^m)$ and $G(x) = o(|x|^m)$.
Therefore we have $f_m(x) = g_m(A(x))$.
Namely, $f_m$ and $g_m$ are linearly equivalent.
\end{proof}

\begin{example}\label{c1moduli}
Let $f_t : (\R^2,0) \to (\R,0)$, $t \in \R$, be a polynomial
function defined by
$$
f_t(x,y) = x^4 + t x^2 y^2 + y^4.
$$
By an elementary calculation, we can see that there are
$a$, $b$, $c$, $d \in \R$ with $ad - bc \ne 0$ such that
$$
(ax + by)^4 + t_1(ax + by)^2(cx + dy)^2 + (cx + dy)^4
= x^4 + t_2x^2y^2 + y^4
$$
if and only if $t_1 = t_2$ or $(t_1 + 2)(t_2 + 2) = 16$.
Therefore it follows from lemma \ref{initialform} that
$f_{t_1}$ and $f_{t_2}$, $t_1$, $t_2 \in \R$, are $C^1$-equivalent
if and only if $t_1 = t_2$ or $(t_1 + 2)(t_2 + 2) = 16$. 
\end{example}





\bigskip
\section{Construction of  bi-Lipschitz and $C^1$ invariants}
\label{invariantbi-L}
\medskip

Let $f(x,y)$ be a real analytic two variable function germ: 
\begin{equation}\label{expansion}
f(x,y) = f_m(x,y) + f_{m+1}(x,y) + \cdots , 
\end{equation}
where $f_j$ denotes the $j$-th homogeneous form of $f$.   
We say that $f$ is \emph{mini-regular in $x$} 
if  $ f_m(1,0) \ne 0$.  
Unless otherwise specified we shall always assume that 
 the real analytic function germs  are mini-regular in $x$.

We shall consider the demi-branches of real analytic arcs at   
$0\in \R^2$ of the following form
\begin{equation*}
\gamma : x = \lambda (y) = a_1 y^{n_1 / N} + a_2 y^{n_2 / N}
+ \cdots , \quad y\ge 0,
\end{equation*}
where $\lambda (y)$ is a convergent fractional power series, $N$ and $ n_1 < n_2 < \cdots$ 
are positive integers having no common divisor, $a_i\in \R$, $N, n_i \in \N$.  
We shall call such a 
demi-branch  \emph{allowable}  if  $n_1/N \ge 1$, that is $\gamma$ is transverse to the 
$x$-axis.

Given $f$ and $\gamma$ as above.  We define
 {\em the order function of $f$ relative to}
$\gamma$, $\ord_{\gamma}f : [1,\infty ) \to \R$ as follows.  Fix  $\xi \ge 1$ and expand 
\begin{equation}\label{genericarc}
f(\lambda (y)+zy^{\xi},y) 
= P_{f,\gamma,\xi} (z) y^{\ord_{\gamma}f(\xi )} + \cdots  , 
\end{equation}
where the dots denote higher order terms in $y$ and  $\ord _{\gamma}f(\xi)$ 
is the smallest exponent with non-zero coefficient.  
This coefficient,  
$P_{f,\gamma,\xi} (z)$,  is a  polynomial function of $z$.

{\em By the Newton polygon of $f$ relative
to} $\gamma$, denoted by $NP_{\gamma}f$, we mean 
the Newton polygon of $f(X+\lambda (Y),Y)$
(cf. \cite{kuoparusinski1}).
Its boundary, called the {\em Newton boundary} and 
denoted by $NB_{\gamma}f$,  is the union of compact faces of $NP_{\gamma}f$.

\begin{rem}
Both the Newton boundary $NB_{\gamma}f$ and the order function  
$\ord_{\gamma}f : [1,\infty ) \to \R$ 
depend only on $f$ and on the demi-branch $\gamma$ considered as a set
germ at the origin.  They are independent of the choice of local
coordinate system, 
as long as $f$ is mini-regular in $x$ and $\gamma$ is 
allowable.  This follows from  corollary \ref{lipinvariant1}.  

As for $P_{f,\gamma,\xi}$, it depends on the choice of coordinate
system,
 but only on its linear part, 
see corollary \ref{c1invariant1} and 
proposition \ref{c1general} below.   
\end{rem}



\begin{prop}\label{legendre}
The Newton boundary $NB_{\gamma}f$ determines the
order function $\ord_{\gamma}f$ and vice versa.
More precisely, let $\varphi : (0,m] \to [0, \infty ]$ be the 
piecewise linear function whose graph $y = \varphi (x)$ is $NB_{\gamma}f$,
then we have 
\begin{eqnarray*}
& \varphi(x) & = 
\max_{\xi} \ (\ord_{\gamma}f(\xi ) - \xi x) \ \ \
(Legendre \ transform), \\
& \ord_{\gamma}f(\xi ) & = 
\min_{x} \ (\varphi (x) + \xi x) \ \ \
(inverse \  Legendre \ transform).
\end{eqnarray*}  
\end{prop} 

\begin {proof}  
Let $f(X+\lambda (Y),Y) = \sum_{i,j} c_{ij} X^i Y^j$.
Then, see the picture below,  
$$
\ord_{\gamma} f(\xi ) = \min_{i,j} \{ j + i\xi ; 
\ c_{i,j} \ne 0 \} = \min_x \{ \varphi (x) + x\xi \}, 
$$
that shows the second formula.  

\vspace{3mm}
\epsfxsize=4cm
\epsfysize=3.5cm
\begin{equation*} 
\epsfbox{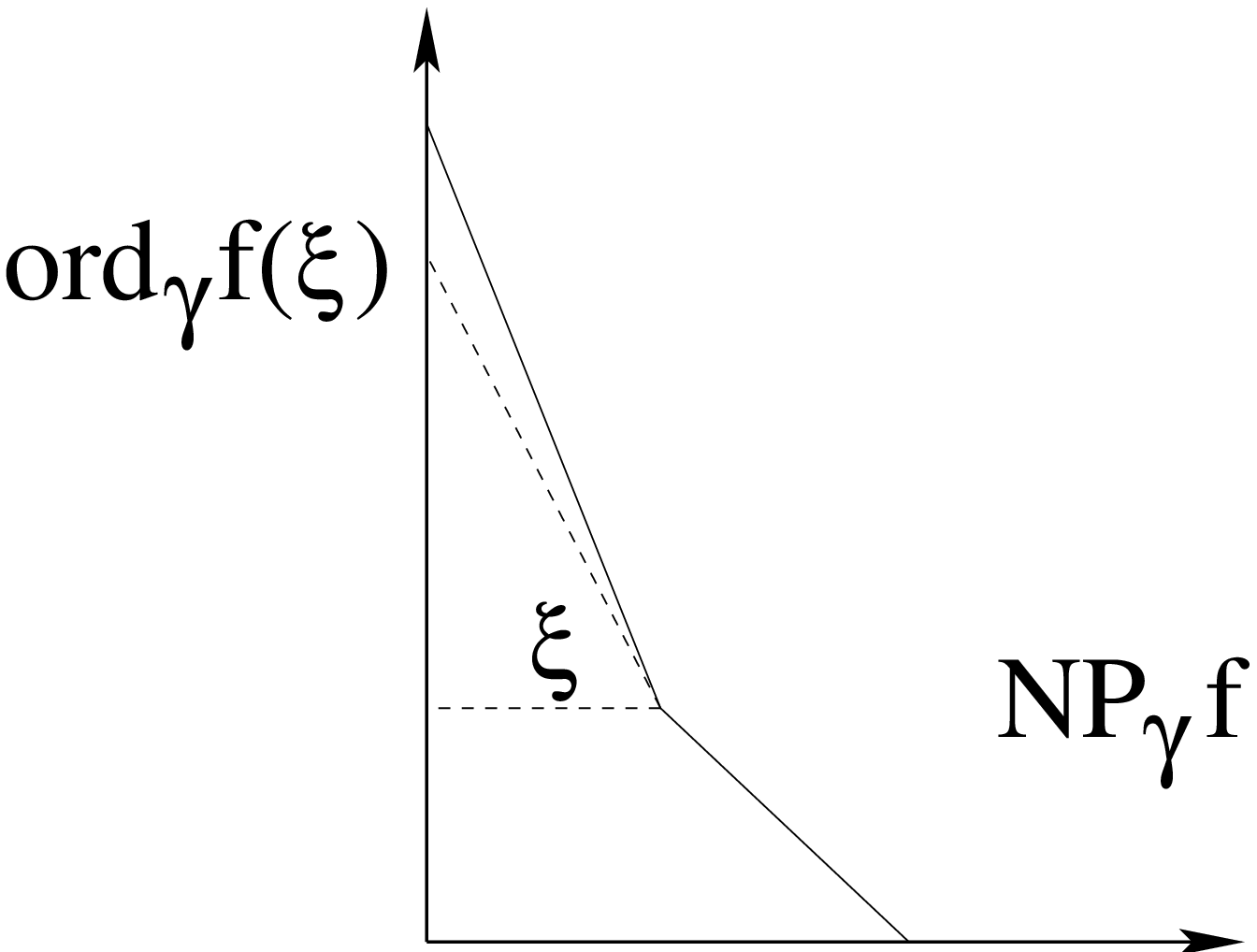}\end{equation*}

\noindent 
The first formula follows from the second one.
\end{proof}

The following example illustrates the meaning of proposition \ref{legendre}.

\begin{example}\label{orderfunction}
Let $f(x,y) = x^2 - y^3$, and let $\gamma_1 : x = y^{\frac 3 2}$
and $\gamma_2 : x = y^{\frac 3 2} + y^{\frac 5 2}$.
Then 
\begin{eqnarray*}
& & f(y^{\frac 3 2} + z y^{\xi},y) 
= 2z y^{{\frac 3 2}+\xi} +  z^2y^{2\xi}, \\ 
& & f(y^{\frac 3 2} + y^{\frac 5 2} + z y^{\xi},y) 
= 2y^4 + y^5 + 2z y^{{\frac 3 2}+\xi} + 2z y^{{\frac 5 2}+\xi} +  z^2y^{2\xi} . 
\end{eqnarray*}
Therefore the order functions of $\gamma_1$ and $\gamma_2$ are given by 
\begin{equation*}
   \ord_{\gamma_1} f(\xi ) = \ \begin{cases}
                     \ 2\xi & \text{for $1 \le \xi \le {\frac 3 2}$} \\
                     \ {\frac 3 2} + \xi  & \text{for $\xi \ge {\frac 3 2}$}
            \end{cases} 
\end{equation*}

\begin{equation*}
   \ord_{\gamma_2} f(\xi ) = \ \begin{cases}
                       \ 2\xi & \text{for $1 \le \xi \le {\frac 3 2}$} \\
                       \ {\frac 3 2} + \xi & \text{for ${\frac 3 2} \le \xi \le {\frac 5 2}$} \\
                       \ 4 & \text{for $\xi \ge {\frac 5 2}.$}
         \end{cases} 
\end{equation*}

\vspace{4mm}
\epsfxsize=12cm
\epsfysize=3.5cm
\begin{equation*} 
\epsfbox{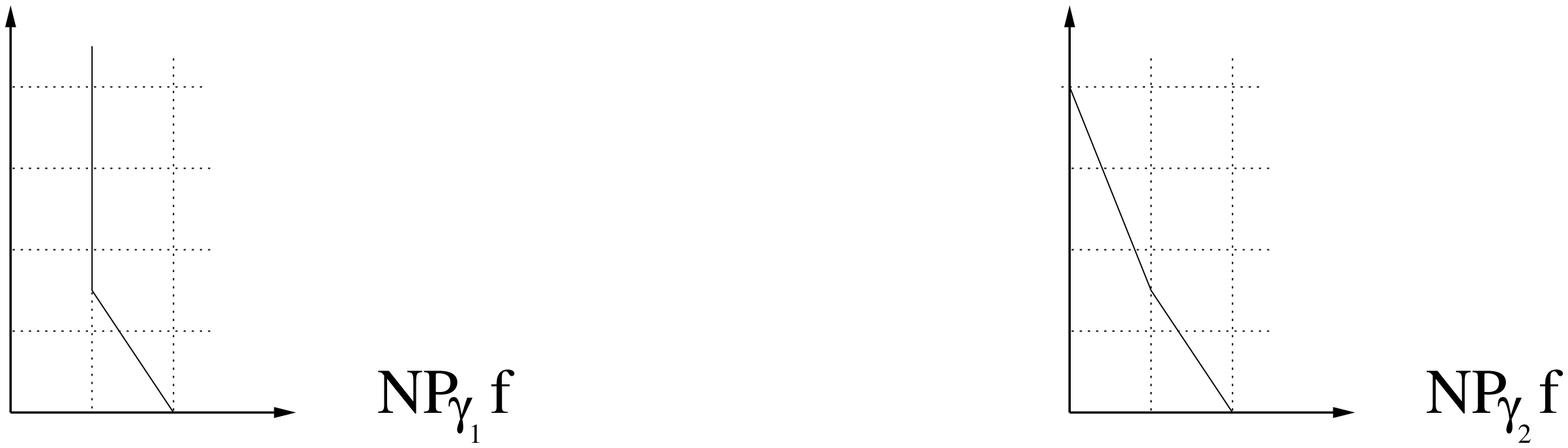}\end{equation*}

Next  we compute the Newton boundaries: $f(X+Y^{\frac 3 2},Y) = X^2 + 2XY^{\frac 3 2}$
and $f(X+Y^{\frac 3 2}+Y^{\frac 5 2},Y) = X^2 + 2XY^{\frac 3 2}
+ 2XY^{\frac 5 2} + 2Y^4 + Y^5$.  Therefore 

\begin{equation*}
   \varphi_1 (x) = \ \begin{cases}
                     \ \infty & \text{for $0 < x < 1$} \\
                     \ 3 - {\frac 3 2}x & \text{for $1 \le x \le 2$}
            \end{cases} 
\end{equation*}

\begin{equation*}
   \varphi_2 (x) = \ \begin{cases}
                     \ 4 - {\frac 5 2}x & \text{for $0 < x < 1$} \\
                     \ 3 - {\frac 3 2}x & \text{for $1 \le x \le 2.$}
            \end{cases} 
\end{equation*}

\end{example}
\medskip

The Newton boundary and the order function give rise to some invariants of bi-Lipschitz and 
$C^1$ equivalences of two variable real analytic function germs.  
We shall introduce them below.  For the 
$C^1$ equivalence we first consider the equivalence given by $C^1$-diffeomorphisms 
$\sigma : (\R^2,0) \to (\R^2,0)$ with $D\sigma (0) =\id$.  The general case will 
be treated in section \ref{arbitraryc1}.   

For  an allowable real analytic demi-branch $\gamma: x=\lambda (y)$ we define  
the {\em horn-neighbourhood of $\gamma$ 
with exponent} $\xi \ge 1$ {\em and width} $N > 0$ by
$$
H_{\xi}(\gamma ;N) := \{ (x,y); \ |x - \lambda (y)| \le
N|y|^{\xi}, y>0 \} .
$$

\medskip
\begin{prop}\label{lipinvariant2}
Let $\sigma : (\R^2,0) \to (\R^2,0)$ be a bi-Lipschitz homeomorphism,
and let $f$, $g : (\R^2,0) \to (\R,0)$ be real analytic function germs
such that $f = g \circ \sigma$.
Suppose that $\gamma$, $\tilde{\gamma}$ are allowable real analytic demi-branches and that  
there exist $ \xi_0 \ge 1$ and $N > 0$ such that
 $$
\sigma (\gamma) \subset H_{\xi_0}(\tilde{\gamma} ;N) . 
$$
Then, for $1 \le \xi \le \xi_0$, $\ord_{\gamma}f(\xi ) = \ord_{\tilde{\gamma}}g(\xi )$ and 
$\deg P_{f,\gamma,\xi} =
\deg P_{g,\tilde \gamma,\xi}$.
\end{prop}

\begin{proof}
If $\xi =1$ then $\ord_{\gamma}f(\xi ) = \deg P_{f,\gamma,\xi} =
\mult_0 f$ 
and the claim follows from bi-Lipschitz invariance of multiplicity.  

Suppose that $\xi_0 >1$ and $1 < \xi \le \xi_0$.    
Let $\sigma (x,y) = (\sigma_1(x,y), \sigma_2 (x,y))$, $\gamma: x=\lambda (y)$, 
$\tilde \gamma: x= \tilde \lambda (y)$.  Then, there exists $C>0$ such that 
$\tilde y(y) := \sigma_2 (\lambda(y),y)$ 
satisfies
\begin{equation*}
\frac 1 C y \le \tilde y(y) \le Cy. 
\end{equation*}

\begin{lem}\label{inclusion}
 For any  $1 < \xi \le \xi_0$ and 
$M>0$ there is $\tilde M_\xi$ such that 
$$\sigma (H_{\xi}(\gamma ;M))
\subset H_{\xi}(\tilde{\gamma} ;\tilde M_\xi ). $$
Moreover, $\tilde M_\xi$ can be chosen of the form 
$\tilde M_\xi= A M $ if $\xi <\xi_0$ and  
$\tilde M_{\xi_0} = A M +N $ if $\xi =\xi_0$. 
\end{lem}

\begin{proof}  
By Lipschitz property, for $(x,y) \in H_{\xi}(\gamma ;M)$ near
$0 \in \R^2$,
\begin{eqnarray*}
& & |\sigma_2(x,y) - \tilde y| = |\sigma_2(x,y) -  \sigma_2 (\lambda(y),y)| 
\le LMy^\xi \le LMC^\xi \tilde y^\xi = o(\tilde y)  , 
\end{eqnarray*}
and 
\begin{eqnarray*}
& |\sigma_1(x,y) - \tilde \lambda ( \tilde y)| & \le  |\sigma_1(x,y) -  \sigma_1 (\lambda(y),y)|
+   |\sigma_1(\lambda (y),y) - \tilde \lambda ( \tilde y)| \\
&&  \le  LMC^\xi \tilde y^\xi  + N \tilde y^{\xi_0} . 
\end{eqnarray*}
Finally, for an arbitrary $\varepsilon >0$, there is a neighbourhood
$U_{\epsilon}$ of $0 \in \R^2$ such that for
$(x,y) \in H_{\xi}(\gamma ;M) \cap U_{\epsilon}$,
\begin{eqnarray*}
& |\sigma_1(x ,y) - \tilde \lambda ( \sigma_2(x ,y))| & 
 \le  |\sigma_1(x ,y) -   \tilde \lambda ( \tilde y)|
+   |\tilde \lambda ( \tilde y)- \tilde \lambda ( \sigma_2(x ,y))| \\
& &  \le  LMC^\xi \tilde y^\xi  + N \tilde y^{\xi_0}  + (\tilde \lambda '(0) + \varepsilon) LMC^\xi \tilde y^\xi . 
\end{eqnarray*}
\end{proof}

Since $\sigma $ is bi-Lipschitz it can be shown by a similar argument that there exists $N'$ for which 
$\sigma (H_{\xi}(\gamma ;N')) \supset \tilde \gamma$, that is $\sigma
\inv  (\tilde \gamma) \subset H_{\xi_0}({\gamma} ;N') $.  
Thus the assumptions of Proposition \ref{lipinvariant2} are symmetric with respect to 
$f$ and $g$.  

Let $x=\delta (y) = \lambda (y) + cy^\xi$, $c$ arbitrary.  On one hand, by lemma \ref{inclusion}, 
$$
|g(\sigma(\delta (y),y)) | \le 
 \max \{P_{g,\tilde \gamma,\xi}(z) ; |z|\le A|c|+N+1 \} \, \tilde y^{\ord_{\tilde \gamma}  g(\xi)} .
$$
On the other hand 
$$
g(\sigma(\delta (y),y)) = f ( \lambda (y) + cy^\xi,y) = P_{f,\gamma,\xi}(c) y^{\ord_{\gamma}   f (\xi)} + \cdots .
$$
This implies  $\ord_{\gamma}f(\xi) \ge \ord_{\tilde{\gamma}}g (\xi)$, then by symmetry 
 $\ord_{\gamma}f(\xi) = \ord_{\tilde{\gamma}}g (\xi)$, and finally $\deg P_{f,\gamma,\xi} =
\deg P_{g,\tilde \gamma,\xi}$.
\end{proof}

\begin{cor}\label{lipinvariant1}
Let $\sigma : (\R^2,0) \to (\R^2,0)$ be a bi-Lipschitz homeomorphism,
and let $f$, $g : (\R^2,0) \to (\R,0)$ be real analytic function germs
such that $f = g \circ \sigma$.
Suppose that $\gamma$, $\tilde{\gamma}$ are allowable real analytic demi-branches 
 such that $\sigma (\gamma ) = \tilde{\gamma}$
as set-germs at $(0,0)$. 
Then, for all $\xi \ge 1$, $\ord_{\gamma}f(\xi ) = \ord_{\tilde{\gamma}}g(\xi )$ and 
$\deg P_{f,\gamma,\xi} =
\deg P_{g,\tilde \gamma,\xi}$.
In particular, $ \NB_{\gamma}f = \NB_{\tilde{\gamma}}g$.  
\end{cor}


\smallskip
\begin{prop}\label{c1invariant1} 
Let $\sigma : (\R^2,0) \to (\R^2,0)$ be a $C^1-$diffeomorphism with $D\sigma (0) = \id$,
and let $f$, $g : (\R^2,0) \to (\R,0)$ be real analytic function germs
such that $f = g \circ \sigma$.
Suppose that $\gamma$, $\tilde{\gamma}$ are allowable real analytic demi-branches such that  
$\sigma (\gamma ) = \tilde{\gamma}$
as set-germs at $(0,0)$. Then for all $\xi \ge 1$ 
$$ 
P_{f,\gamma, \xi} = P_ {g, \tilde \gamma, \xi} . 
$$
\end{prop}

\begin{proof}
It is more convenient to work in a wider category and assume that $f$ and $g$ are convergent 
fractional power series of the form 
\begin{equation}\label{fractional}
\sum_{(i,j) \in \N \times \frac 1 q  \N} c_{ij} x^i y^{j}, 
\end{equation}
where $c_{ij}\in \R$, $q\in \N$.  Such series give rise to  function germs well-defined on $y\ge 0$.   
Let $\gamma : x=\lambda (y)$,  $\tilde{\gamma}: x=\tilde \lambda (y)$. 
Then 
$$
f\circ H_1 = g \circ H_2\ \circ (H_2\inv  \circ \sigma \circ H_1), 
$$
where  $H_1 (x,y) = (x+\lambda (y),y)$ and $H_2 (X,Y) = (X+\tilde \lambda (Y),Y)$.  
Then $\tilde f = f\circ H_1$, $\tilde g =g \circ H_2$ are fractional power series.   
The map $\tilde \sigma = H_2\inv \circ \sigma \circ H_1$ is $C^1$ and $D\tilde \sigma (0) = \id$.  
Thus by replacing $f, g, \sigma$ by $\tilde f, \tilde g, \tilde \sigma$ we may suppose that 
$\lambda \equiv \tilde \lambda \equiv 0$, that is the image of the $y-$axis is the $y-$axis.  

If $\sigma$ preserves the $y-$axis then it is of the form 
$$
\sigma (x,y) = (x\varphi (x,y), y+ \psi (x,y)),
$$
with $\varphi (x,y),  \psi (x,y)$ continuous and $\varphi(0,0)=1, \psi (0,0)=0$.  

Let  $g(x,y)$ be a fractional power series as in \eqref{fractional}.  
The expansion \eqref{genericarc} 
still holds for $g$ and any allowable demi-branch.  We use this
property for the (positive) $y-$axis 
as a demi-branch that we denote 
below  by $\underline 0$ (since it is given by $\lambda \equiv 0$).

\begin{lem}\label{lemma1} 
Let  $g(x,y)$ be a fractional power series as in \eqref{fractional}. 
 Then for all $\alpha (y), \beta (y)$ such that  
$\alpha (y) = o(y), \beta (y) = o(y)$,  $\xi \ge 1$, and $z\in \R$ bounded  
$$ 
g( (z + \alpha (y))  y^\xi , y + \beta (y)) = 
P_{g,\underline 0,\xi}(z) y ^{\ord_{\underline 0} g (\xi)} + 
 o(y ^{\ord_{\underline 0} g (\xi)}).  
$$
\end{lem}

\begin{proof}
We have 
\begin{equation*}
g(zy^{\xi},y) 
= P_{g,\underline 0,\xi} (z) y^{\ord_{\underline 0}g(\xi )} + o(y^{\ord_{\underline0}g(\xi )} )  .  
\end{equation*}
More precisely 
$g(zy^{\xi},y)  - P_{g,\underline 0,\xi} (z) y^{\ord_{\underline 0}g(\xi )} \to 0 $ as $y\to 0$ and $z$ is bounded.  
Then 
\begin{eqnarray*}
g((z + \alpha (y))  y^\xi , y + \beta (y)) & = &   g((z+ \alpha (y)) (y+ \beta (y) - \beta (y)) ^\xi , y + \beta (y)) \\
&=  &   g((z+  \tilde \alpha(y)) (y+ \beta (y) ) ^\xi , y + \beta (y))  \\
& = & P_{g,\underline 0,\xi} ( z+  \tilde \alpha(y)) (y + \beta (y))^{\ord_{\underline 0}g(\xi )} + 
o((y)^{\ord_{\underline 0}g(\xi )}) \\
& = &   P_{g,\underline 0,\xi}(z) y ^{\ord_{\underline 0} g (\xi)} + o(y ^{\ord_{\underline 0} g (\xi)}),  
\end{eqnarray*}
where $\tilde \alpha(y) = o(y)$.  
\end{proof}

To complete the proof of proposition \ref{c1invariant1} we apply lemma \ref{lemma1} to $g$, 
$\alpha (y) = \varphi (cy^\xi , y)-1$, and $\beta (y) =  \psi (cy^\xi , y)$, where $c\in \R$ is a constant.  Then 
\begin{eqnarray*}
f(cy^\xi , y) & = & g\circ \sigma (cy^\xi, y) = g (c y^\xi \varphi (cy^\xi
, y), y + \psi (cy^\xi , y))\\
& =  &  
P_{g,\underline 0,\xi}(c) y ^{\ord_{\underline 0} g (\xi)} + o(y ^{\ord_{\underline 0} g (\xi)}) . 
\end{eqnarray*}
Therefore, by expanding 
$f(cy^\xi , y) = P_{f,\underline 0,\xi}(c) y ^{\ord_{\underline 0} f
  (\xi)} + o(y ^{\ord_{\underline 0} f (\xi)})  $, we obtain 
$$
P_{f,\underline 0,\xi}(c) y ^{\ord_{\underline 0} f (\xi)} + o(y ^{\ord_{\underline 0} f (\xi)})  = 
P_{g,\underline 0,\xi}(c) y ^{\ord_{\underline 0} g (\xi)} + o(y ^{\ord_{\underline 0} g (\xi)}) , 
$$
that shows  $ P_{f,\underline 0,\xi} = P_{g,\underline 0,\xi}  $.  
This ends the proof of proposition 
\ref{c1invariant1}.
\end{proof}

\medskip
\begin{prop}\label{c1invariant2} 
Let $\sigma : (\R^2,0) \to (\R^2,0)$ be a $C^1-$diffeomorphism with $D\sigma (0) = \id$,
and let $f$, $g : (\R^2,0) \to (\R,0)$ be real analytic function germs
such that $f = g \circ \sigma$.
Suppose that $\gamma$, $\tilde{\gamma}$ are allowable real analytic demi-branches and 
that there exist $ \xi_0 \ge 1, N > 0$ such that
$$
\sigma (\gamma ) \subset 
H_{\xi_0}(\tilde{\gamma};N) . 
$$
 Then for all $1\le \xi < \xi_0$,  
$$ 
P_{f,\gamma, \xi} = P_ {g, \tilde \gamma, \xi}.   
$$
Moreover,  $ P_{f,\gamma, \xi_0}$ and $ P_ {g, \tilde \gamma, \xi_0} $ 
have the same degrees and their leading coefficients coincide.
\end{prop}

\begin{proof}
Let $\xi_0>1$.  Then the tangent directions at the origin to $\gamma$
and $\tilde \gamma$ 
coincide.  We assume as above that $\gamma$, and $\tilde \gamma$
resp., is the (positive) $y-$axis.  
Write 
$$
\sigma (x,y) = (\sigma_1(0,y) + x \varphi (x,y), y+ \psi (x,y)),
$$
with $\varphi (x,y),  \psi (x,y)$ continuous and $\varphi(0,0)=1, \psi
(0,0)=0$.   The assumption on the image of 
$\gamma$ gives 
$$
|\sigma_1 (0,y) | \le N_1 |y|^{\xi_0}, 
$$
with $N_1 = N + \varepsilon$.   Then, for $\xi <\xi_0$, $\sigma_1 (0,y) = o(y^\xi)$ and 
\begin{eqnarray*}
f(cy^\xi , y) = g\circ \sigma (cy^\xi, y)  & = &  g (c y^\xi \varphi (cy^\xi , y  ) 
+ \sigma_1 (0,y), y + \psi (cy^\xi , y)) \\
& = & g( (c + \alpha (y))  y^\xi , y + \beta (y)) 
\end{eqnarray*}
with $\alpha (y) = o(y), \beta (y) = o(y)$.  
Thus the first claim follows again from lemma \ref{lemma1}.

If $\xi = \xi_0 >1$ then the same computation shows that 
$$
P_{f,\underline 0,\xi} (c) \in \{ P_{g,\underline 0,\xi} (z) ; \
|z - c| \le N_1 \} ,
$$
for all $c$.  That shows that the degrees of $P_{f,\underline 0,\xi} $
and $ P_{g,\underline 0,\xi}$ and 
their leading coefficients coincide.  

If $\xi_0=1$ then $P_{f,\gamma, 1}$  depends only on the initial homogeneous form of $f$, 
denoted by $f_m$ in \eqref{expansion}, and the tangent direction to $\gamma$ at the origin.  
Then  $m = \deg P_{f,\gamma, 1}$ 
and the leading coefficient of $P_{f,\gamma, 1}$ is independent of the choice of $\gamma$.  
But the initial homogeneous forms of $f$ and $g$ coincide by 
  lemma \ref{initialform}.  This completes the proof of proposition \ref{c1invariant2}.  
\end{proof}

Let $f(X+\lambda (Y),Y) = \sum_{i,j} c_{ij} X^i Y^j$.  Then by  {\em the initial Newton polynomial of} $f$
{\em relative to } $\gamma$,  we mean 
\begin{equation}
in_{\gamma}f = \sum_{(i,j) \in NB_{\gamma}(f)} c_{ij} X^i Y^j.  
\end{equation}
Note that $in_{\gamma}f$ is a fractional polynomial, $(i,j) \in \Z\times \Q, i\ge 0, j \ge 0$.

\begin{cor}
Let $\sigma : (\R^2,0) \to (\R^2,0)$ be a $C^1-$homeomorphism with $D\sigma (0) = \id$,
and let $f$, $g : (\R^2,0) \to (\R,0)$ be real analytic function germs
such that $f = g \circ \sigma$.
Suppose that $\gamma$, $\tilde{\gamma}$ are allowable real analytic demi-branches such that  
$\sigma (\gamma ) = \tilde{\gamma}$
as set-germs at $(0,0)$.  Then $in_{\gamma}f = in_{\tilde \gamma} g $.  
\end{cor}

\begin{proof}
Let $\Gamma_\xi $ be a compact face of $NB_{\gamma}(f)$ of slope $-\xi$.  Then 
\begin{equation*}
P_{f,\gamma,\xi} (z)  = \sum_{(i,j) \in \Gamma_\xi} c_{ij} z^i  .  
\end{equation*}
\end{proof}


\subsection{C$^ 1$ invariance of Puiseux pairs of roots}\label{Puiseuxroots}

Let $\gamma: x= \lambda (y)$ be an allowable real analytic demi-branch.   
The Puiseux pairs of $\gamma$ 
are pairs of relatively prime positive integers 
$(n_1,d_1), \ldots, (n_q,d_q)$, $d_i>1$ for  $i=1,\ldots,q$,   
$\frac {n_1}{d_1} < \frac {n_2}{d_1 d_2} < \cdots < 
\frac {n_q }{d_1\ldots d_q}$, 
such that  
\begin{eqnarray}\label{Ppairs} 
& & \qquad  \qquad  \qquad   \lambda (y)  =   \sum_\alpha a_\alpha y^\alpha = \\ 
\notag
&  & = \sum_{j=1}^{[{n_1}/{d_1}] }  a_j y^j  
+ \sum_{j=n_1}^{[{n_2}/{d_2}] }   a_{{j}/{d_1}}y^{ {j}/{d_1}} 
+ \sum_{j=n_2}^{[{n_3}/{d_3}] } a_{ {j}/{d_1 d_2}}y^{ \frac {j} {d_1 d_2}} + \cdots 
+ \sum_{j=n_q}^\infty   a_{{j}/{d_1 d_2 \cdots d_q}}y^{ \frac {j} {d_1 d_2 \cdots d_q }} 
\end{eqnarray}
and $a_{n_i/d_1\cdots d_i}\ne 0$ for $i=1,\ldots ,q$, cf. e.g. \cite{wall}.   
The exponents $n_i/d_1\cdots d_i$ will be called \emph{the (Puiseux)
characteristic exponents of 
$\gamma$}.   
The corresponding coefficients 
$A_i (\gamma ):= a_{n_i/d_1\cdots d_i}$ for $i=1,\ldots ,q$ 
will be called 
\emph{the characteristic coefficients of $\gamma$}.

\begin{prop}\label{Puiseuxinvariance1}
Let $\sigma : (\R^2,0) \to (\R^2,0)$ be a $C^1$ diffeomorphism 
and let $f$, $g : (\R^2,0) \to (\R,0)$ be real analytic function germs
such that $f = g \circ \sigma$. Suppose that $\gamma \subset f\inv (0)$, $\tilde{\gamma}\subset g\inv (0)$
 are allowable real 
analytic demi-branches such that  $\sigma (\gamma) = \tilde \gamma$ as set germs.  
Then the Puiseux characteristic pairs of $\gamma$ and 
$\tilde \gamma$ coincide.

Moreover, if $D\sigma (0)$ preserves orientation then the signs of characteristic coefficients 
of $\gamma$ and $\tilde \gamma$ coincide. 
\end{prop}

\begin{proof} 
Let us write $f$ and $g$ as in Lemma \ref{initialform},
and let $A$ be the linear transformation of $\R^2$ given in the proof.
Define $\tilde{g} : (\R^2,0) \to (\R,0)$ by
$\tilde{g}(x) = g(A(x))$.
Then $\tilde{g}$ has the form
$$
\tilde{g}(x) = f_m(x) + \tilde{g}_{m+1}(x) + \tilde{g}_{m+2}(x) 
+ \cdots
$$
where $\tilde{g}_j(x) = g_j(A(x))$, $j = m+1$, $m+2, \cdots$.
Then $\tilde{g}$ is linearly equivalent to $g$
and $C^1$-equivalent to $f$.
Therefore we may assume that $D\sigma (0) = Id$.

Let $(n_1,d_1), \ldots,$ $(n_{i},d_{i})$ be the first $i$ Puiseux pairs 
of $\gamma$  and  fix  $\xi =  {n}/ {d_1 \cdots d_{i}d}$, $\gcd (n,d)=1$, 
$d>1$, such that $\xi >  
\frac {n_i }{d_1\ldots d_{i}}$.   We may suppose by the inductive assumption that 
$(n_1,d_1), \ldots,$ $(n_{i},d_{i})$ are also the first $i$ Puiseux pairs 
of $\tilde \gamma$.  

Consider the truncation $\gamma_\xi$ of $\gamma$ at $\xi$: if 
$\gamma : x= \lambda (y) =\sum a_\alpha y^\alpha$ then 
$$\gamma_\xi: x= \lambda_\xi (y) = 
\sum _{\alpha <\xi} a_\alpha y^\alpha .
$$   
Denote by $P(z) = P_{f,\gamma,\xi}$, $P_0(z) = P_{f,\gamma_\xi ,\xi}$ the polynomials defined 
for $\gamma$ and $\gamma_\xi$ by the expansion \eqref{genericarc}.  Then $P(z) = P_0 (z+a_\xi)$.   
Denote $\ord = \ord _{\gamma_\xi} f (\xi)=  \ord _{\gamma} f (\xi)$.

\begin{lem}\label{Puiseux}
There exist an integer $k\ge 0$ and a  polynomial $\tilde P_0$ such that 
$ P_0 (z) =  z^k \tilde P_0(z^d)$.  
\end{lem}

\begin{proof}
Write 
\begin{eqnarray*}
f(\lambda_\xi (y) + x,y)    =   
\sum _{\xi i + j\ge  \ord } c_{ij} x^i y^j  
= P_0 (x/y^\xi) y^{\ord}  +  
\sum _{\xi i + j>  \ord} c_{ij} x^i y^j . 
\end{eqnarray*}
Since the Puiseux pairs of $\gamma_\xi$ determine the possible 
denominators of the exponents of $y$ in $f(\lambda_\xi (y) + x,y) $ 
\begin{equation} \label{condition}
c_{ij} \ne 0 \Rightarrow j (d_1 \cdots d_{i}) \in \N.
\end{equation} 
Let $i_0 = \deg P_\xi$, $j_0 =  \ord  - \xi i_0$.  Then 
$P_0 (z) = c_{i_0j_0} z^{i_0} +$ lower degree terms.

We show that for $(i,j)$ such that $\xi i + j=  \ord$, the condition 
$  j (d_1 \cdots d_{i}) \in \N$ of \eqref{condition} is equivalent to
$(i-i_o)\in d\N$.  
Indeed, since $\xi(i-i_0) + (j-j_0) =0$ and $j_0 (d_1 \cdots d_{i}) \in \N.$,   
\begin{eqnarray*}
& & j  (d_1 \cdots d_{i})   \in \N \Leftrightarrow 
\xi (i-i_0) d_1 \cdots d_{i}\in \N \Leftrightarrow (i-i_0) n/d \in \N 
\Leftrightarrow (i-i_0)\in d\N .
\end{eqnarray*}
Thus $P_0 = a_{i_0j_0} z^k \hat P_0(z^d)$ with $\hat P_0$ unitary.   
\end{proof}

By lemma \ref{Puiseux} 
$$
P(z) = P_0 (z+a_\xi)=   
(z+ a_\xi) ^k \tilde P_0((z+a_\xi)^d) = A_0 z^{i_0} + A_1 z^{i_0-1} + \cdots.    
$$
If $\gamma \subset f\inv (0)$ then $P(0)=0$ and therefore $P$ is
not 
identically equal to a constant. If $P$ is not a constant then we may compute  $a_\xi =
\frac{A_1}{i_0A_0}$.  Consequently, by proposition \ref{c1invariant1}, 
$\xi$ is a Puiseux characteristic exponent of $\gamma$ iff it is a one
of $\tilde \gamma$.  
Moreover, the characteristic coefficients are 
the same.  (Arbitrary linear isomorphisms may change these
coefficients 
but not their signs if the orientation is preserved.)
\end{proof}



\bigskip
\section{$C^1$  equivalent germs are blow-analytically equivalent}
\label{$C^1$}
\medskip

In this section we show theorem \ref{c1givesblow}.  
The proof is based on the characterisation (3) of  theorem \ref{allequivalent}.   
First we recall briefly the construction
 of real tree model,  for the details see  \cite{koikeparusinski2}.   
 
 In \cite {kuolu} Kuo and Lu 
 introduced a tree model $T(f)$ of a complex analytic function germ $f(x,y)$.  
This model allows one to visualise the numerical data given by the contact orders between 
the Newton-Puiseux roots of $f$,  in particular their Puiseux pairs, and determines the minimal 
resolution of $f$.  The real tree model of \cite{koikeparusinski2} 
is an adaptation of the Kuo-Lu tree model 
to the real analytic world.  The main differences are the following.  
The Newton-Puiseux roots of $f$ 
\begin{equation}\label{root}
x=\lambda (y)  = a_1 y^{n_1 / N} + a_2 y^{n_2 / N}
+ \cdots 
\end{equation}
are replaced by real analytic demi-branches or their horn neighbourhoods obtained by restricting 
\eqref{root} to $y\ge 0$, and then truncating it the first non-real 
coefficient $a_i$ that is replaced by a symbol $c$ signifying a generic $c\in \R$ 
(in this way we can 
still keep track of the exponent $n_i/N$).  
The later construction can be reinterpreted geometrically 
in the real world 
by taking "root horns".  The Puiseux pairs of the roots, or of the root horns, 
are added to the numerical data of the tree.  Unlike in the 
complex case they can not be computed 
from the contact orders.  Finally,  the signs of  coefficients at the Puiseux characteristic  exponents 
  are marked on the tree.


\subsection{Real tree model of $f$ relative to a tangent direction}

Let $f (x,y)$ be a real analytic function germ.  Fix $v$ a unit vector
of $\R^2$.  
\emph{The tree model 
$\R T_v(f)$ of $f$ relative to $v$ }is defined as follows. 
 Fix any local system of coordinates $x,y$ 
such that: \\
\emph{- $f(x,y)$ is mini-regular in $x$};\\
\emph{- $v$ is of the form  $(v_1,v_2)$ with $v_2 >0$.  }

Let $x=\lambda(y)$ be a Newton-Puiseux root of $f$ of the form
\eqref{root}.  
If $\lambda$ is not real and $a_i$ is the first non-real coefficient we replace this root by 
\begin{equation}\label{truncatedroot}
x=\ a_1 y^{n_1 / N} + a_2 y^{n_2 / N}
+ \cdots + c y^{n_i / N}, \quad c\in \R \quad \text{generic} ,
\end{equation}
where $c$ is a symbol signifying a generic $c\in \R$.  We call \eqref{truncatedroot} 
{\em a truncated root}.  Let $\Lambda_v$ denote the set of real roots
and truncated roots, restricted to $y\ge 0$,  
 that are tangent to $v$ at the origin.

Suppose $\Lambda_v$ non-empty.  We apply the Kuo-Lu construction to $\Lambda_v$.  
We define \emph{the contact order of}  $\lambda_i$ and $\lambda_j$ of $\Lambda_v$ as  
$$
O(\lambda_i,\lambda_j) := \ord_0 \, (\lambda_i - \lambda_j)(y).
$$
Let $h\in \Q$.  We say that $\lambda_i, \lambda_j$ are \emph{congruent modulo}
 $h^+$ if $O(\lambda_i,\lambda_j)>h$.  

Draw a vertical line as the {\it main trunk} of the tree.  
Mark the number $m_v$ of roots in $\Lambda_v$ counted with 
multiplicites alongside
the trunk. 
Let $h_1:= \min \{O(\lambda_i,\lambda_j) | 1\le i,j\le m_v\}$. Then draw a bar, 
$B_1$, on top of the main trunk. Call $h(B_1):= h_1$ the
{\it height} of $B_1$. 

The roots of $\Lambda _v$ with the original coefficient $a_{h_1}$ real are divided into equivalence
classes, called \emph{bunches},  modulo $h_1^+$.  We then represent each equivalence class
by a vertical line segment drawn on top of $B_1$ in the order
corresponding to the order of $a_{h_1}$ coefficients.  
Each is called a {\it trunk}.  If a trunk consists of $s$ roots 
 we say it has {\it  multiplicity} $s$, 
and mark $s$ alongside (if $s=1$ it is usually not marked).   
The other roots of $\Lambda_v$, that is 
those with the symbol $c$ as the coefficient at $y^{h_1}$,   
do not produce a trunk over $B_1$ and disappear at $B_1$.  

Now, the same construction is repeated recursively on each trunk,
getting more bars, then more trunks, etc..
The height of each bar and the multiplicity of a trunk, are defined likewise.  
Each trunk has a unique bar on top of it.  
The construction terminates at the stage where the bar have infinite height, that is 
is on top of a trunk that contain  
a single, maybe multiple, real root of $f$.  

To each bar $B$ corresponds a unique trunk supporting it and a unique bunch of roots $A(B)$ bounded 
by $B$.  
In this way there is a one-to-one correspondence between trunks, bars, 
and bunches.   We denote by $m_B$  the multiplicity of the trunk supporting $B$.  

Whenever a bar $B$ gives a new Puiseux pair to a root of $A(B)$    
 we mark $0$ on $B$.  If a trunk $T'$ growing on $B$ corresponds
 to the roots of $A$ with coefficient $a_{h(B)}=0$, resp. $a_{h(B)}< 0$,  $a_{h(B)}>0$, then 
 we mark $T'$ as growing at $0\in B$, resp. to the left of $0$, to the right of $0$.  
Graphically,  we mark $0\in B$ by identifying it with the point of $B$ 
that belongs to the trunk  supporting $B$.


\smallskip
\subsection{Real tree model of $f$}
The \emph{real tree model 
$\R T(f)$ of $f$} is defined as follow. 
\begin{itemize}
\item 
Draw a bar $B_0$ that is identified with $S^1$.  
 We define $h(B_0)=1$ and call $B_0$ \emph{the ground bar}. 
\item
Grow on $B_0$ all non-trivial $\R T_v(f)$ for $v\in S^1$, keeping the 
clockwise order.
\item
Let $v_1, v_2$ be any two subsequent unit vectors for which $\R T_v(f)$ is nontrivial. 
Mark  the sign of $f$ in the sector between $v_1$ and $v_2$. 
Note that one such sign determines all the other signs between two subsequent  
unit vectors for which $\R T_v(f)$ is nontrivial (passing $v$ changes
this sign if and only if $\Lambda_v$ contains an odd number of roots.) 
\end{itemize}
If the leading homogeneous part $f_m$ of $f$ satisfies 
$f_m\inv (0) =0$ then $B_0$ is the only bar of $\R T(f)$. 

For instance we give below the real tree model of $f(x,y)=x^2-y^3$. 
More examples are presented in section \ref{bilipschitz} below.

\medskip
\epsfxsize=6cm
\epsfysize=3.5cm
$$\epsfbox{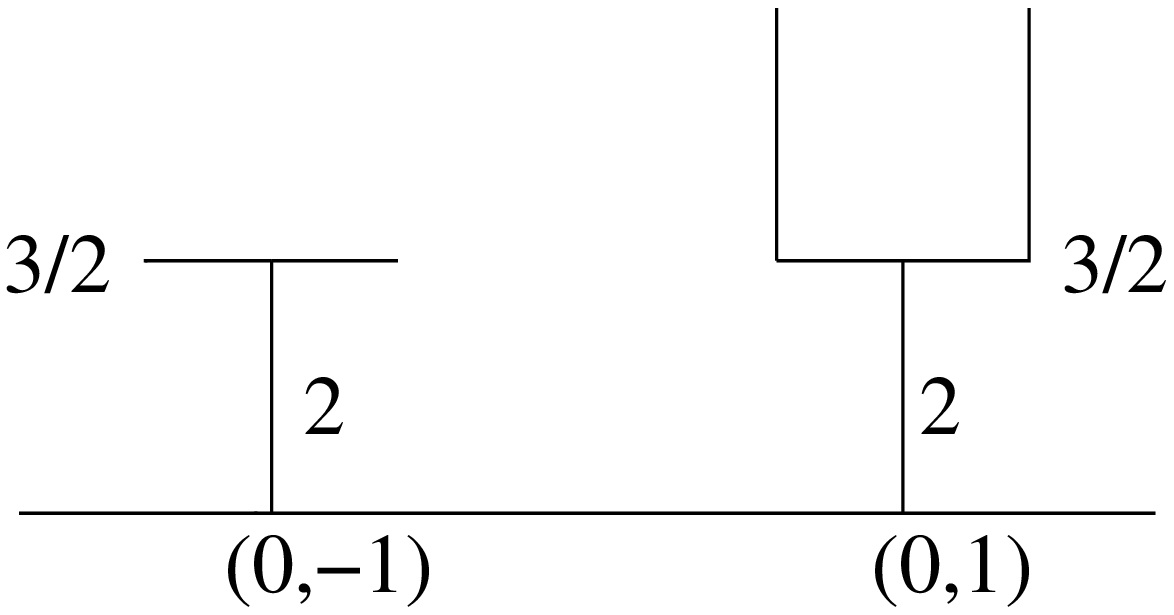}$$


\smallskip
\subsection{Horns}
Recall that for  an allowable real analytic demi-branch $\gamma: x=\lambda (y)$ the
horn-neighbourhood of $\gamma$ 
with exponent $\xi \ge 1$ and width $N > 0$ is given by
$$
H_{\xi}(\gamma ;N) := \{ (x,y); \ |x - \lambda (y)| \le
N|y|^{\xi}, y>0 \} .
$$
We define  {\em the horn-neighbourhood of $\gamma$ of exponent} $\xi$ as 
$H_{\xi}(\gamma ; C)$ for $C$ large and we denote it by $H_{\xi}(\gamma)$.   
\emph{A horn} is a horn-neighbourhood with exponent $\xi >1$.  

If $\gamma_1 : x=\lambda_1(y) $, 
$\gamma_2 : x=\lambda_2(y) $, and $O(\lambda_1,\lambda_2)\ge \xi$ then we identify 
$H_{\xi}(\gamma_1)= H_{\xi}(\gamma_2)$  by meaning that for any $C_1>0$ there is $C_2>0$  such that 
\begin{equation}\label{identify}
H_{\xi}(\gamma_1 ; C_1) \subset H_{\xi}(\gamma_2 ; C_2), \qquad 
H_{\xi}(\gamma_2 ; C_1) \subset H_{\xi}(\gamma_1 ; C_2) .  
\end{equation}

\begin{example}
Let $B$ be a bar  of $\R T_v (f)$, $h(B)>1$.  
Then $B$ defines a horn 
$$
H_B := \{ (x,y); \ |x - \lambda (y)| \le
C|y|^{h(B)} \} ,
$$
where $C$ is a large constant and $x=\lambda (y)$ is any root of bunch $A(B)$.   
\end{example}

\begin{defn}
A horn that equals $H_B$ for a bar $B$ is called \emph{a root horn}.  
\end{defn}

Let $H = H_{\xi}(\gamma),  \gamma: x=\lambda (y)$, be a horn of exponent $\xi$.  Let $\lambda_H(y) $ 
denote the truncation of $\lambda$  at $\xi$, that is $\lambda_H(y) $  is the sum of all terms of $\lambda(y)$ 
of exponent $<\xi$. 
 We define the truncated demi-branch by $\gamma_H : x= \lambda_H(y) $
 and the generic demi-branch   
$\gamma_{H,gen} : x= \lambda_{H,gen}(y)$  by  
\begin{equation}
\lambda_{H,gen } (y) = \lambda_H(y) + c y^{\xi} +  \cdots , \qquad y\ge 0, 
\end{equation} 
where $c\in \R$ is a generic constant.   The \emph {characteristic exponents of $H$} 
are those of $\gamma_{H,gen}$ that are $\le \xi$.   \emph{The signs of 
characteristic coefficients of $H$} are those of $\gamma_{H,gen}$ (or of $\gamma_{H}$)
 corresponding to the exponents $<\xi$.  
Let $\gamma' : x=\lambda' (y)$ be any
 allowable real analytic demi-branch contained in  $H$. 
Then the order function $\ord_{\gamma'} f $, defined by \eqref{genericarc}, restricted to 
$ [1,\xi] $ is independent of the choice of $\gamma'$ and so is the  polynomial  $P_{f,\gamma',\xi'} (z)$ for 
$\xi '<\xi$.   The polynomial $P_{f,\gamma',\xi} (z)$ is independent up to a shift of variable $z$:  if the coefficient 
of $\lambda' (y)$ at $y^{\xi}$ is $a$ then  
$$
P_{f,\gamma',\xi} (z) =P_{f,\gamma_H,\xi} (z+a) .  
$$

\medskip
\begin{prop}\label{conversion}
{\rm [compare  \cite{koikeparusinski2}, Proposition 7.5] } 

Let $H $ be a horn of exponent $\xi$.  Then $H$ is a root horn  for $f(x,y)$ if and only if 
$P_{f,\gamma_H,\xi} (z) $ has at least two distinct complex roots. 

If this is the case, $H=H_B$, then $h(B) = \xi $ and $m_B = \deg P_{f,\gamma_H,\xi_H}$.
\end{prop}

\begin{proof}
Suppose that $H=H_B$ and let $A(B)= \{\gamma_1, \ldots, \gamma_{m_B}
\} $ be the corresponding bunch of roots.  
These roots are truncations of complex Newton-Puiseux roots of $f$: 
\begin{equation}\label{complexroot}
\gamma_{\C,k}: x= \lambda _k (y) = \lambda_H + a_{\xi,k} y^\xi + \cdots, \quad 1\le k\le m_B  
\end{equation}
with $\lambda_H$ real and $a_{\xi,k}\in \C$.  Denote by $\gamma_{\C,j}: x= \lambda _j (y)$, 
$j=m_B +1, \ldots, m,$ the remaining complex Newton-Puiseux roots of $f$.  Then   
\begin{equation*}
f(\lambda_H (y)+zy^{\xi},y)  = u(x,y) \, \prod_{i=1}^m (\lambda_H (y)- \lambda_i (y)+zy^{\xi})  
= P_{f,\gamma_H,\xi} (z) y^{\ord_{\gamma_H}f(\xi )} + \cdots  , 
\end{equation*}
where $u(0,0)\ne 0$.   Note that $O(\lambda_H, \lambda_j) <\xi$ for $j>m_B$.   Therefore 
\begin{eqnarray*}
&  P_{f,\gamma_H,\xi_H} (z) & = u(0,0)  \, \prod_{i=1}^{m(B)} (z -a_{\xi ,i}),  \\
&  \ord_{\gamma_H}f(\xi_H ) & = m_B\xi_H + \sum_{j=m(B)+1}^m O(\lambda_H, \lambda_j). 
\end{eqnarray*}  
By construction of the tree there are at least two roots $\gamma_i$ and $\gamma_j$ of 
\eqref{complexroot} such that 
$O(\lambda_i, \lambda_j)  = \xi$. Thus $P_{f,\gamma_H,\xi_H} (z) $ has
at least two distinct complex roots.  

Let $H = H_{\xi}(\gamma)$ be a horn,  where 
$$
\gamma : x= \lambda_H(y) + a_\xi y^{\xi} +  \cdots. 
$$  
By the Newton algorithm  for computing the complex Newton-Puiseux roots of $f$ 
to each root $z_0$ of $ P_{f,\gamma,\xi}$ of multiplicity $s$ correspond exactly $s$ 
Newton-Puiseux roots of $f$, counted with multiplicities, of the form  
\begin{equation*}
\gamma_0 : x= \lambda_H(y) + (a_\xi +z_0)  y^{\xi} +  \cdots .
\end{equation*} 
(This is essentially the way the Newton-Puiseux theorem is proved as in \cite{walker}.)  
Thus, if $ P_{f,\gamma,\xi}$ has at least two distinct roots, then there exist at least two such Newton-Puiseux 
roots with contact order equal to $\xi$.  This shows that $H$ is of the form $H_B$, as claimed.  
 \end{proof}

Proposition \ref{conversion} shows that for a root horn $H$ of width $\xi_0$, $\deg P_{f,\gamma_H,\xi}> 1$.  
Moreover, for any (Puiseux) characteristic exponent of $\gamma_H $, $\xi <\xi_0$, the 
horn  $H_{\xi}(\gamma_H)$ is a root horn.  Indeed, if $\xi =  {n_i}/ {d_1 \cdots d_{i}}$, then 
$\deg P_{f,\gamma_H,\xi}= d_i$ 
Therefore we may extend the argument of the proof of Proposition 
\ref {Puiseuxinvariance1} to the root horn case.   

\begin{prop}\label{Puiseuxinvariance2}
Let $H=H_{\xi} (\gamma)$ be a horn root.  
Let $\sigma : (\R^2,0) \to (\R^2,0)$ be a $C^1-$diffeomorphism 
and let $f$, $g : (\R^2,0) \to (\R,0)$ be real analytic function germs
such that $f = g \circ \sigma$. 
Suppose that $\gamma$, $\tilde{\gamma}$ are allowable real analytic demi-branches such that  
$$
\sigma (\gamma ) \subset 
H_{\xi}(\tilde{\gamma};N).$$ 
Then the Puiseux characteristic exponents  $H$ and 
$\tilde H$ coincide.

Moreover, if $D\sigma (0)$ preserves orientation then the signs of
characteristic coefficients 
of $H$ and  $\tilde H$ coincide. 
\end{prop} 


\subsection{Characterisation of real tree model in terms of root horns}\label{realconstruction}

The real tree model $\R T(f)$ is determined by the root horns and their numerical invariants, cf. \cite{koikeparusinski2} subsection 7.3. 
The root  horns are ordered by inclusion and by  clockwise order  
around the origin.  Thus $H_B$ is contained in $H_{B'}$ if and only if 
the bar $B$ grows over $B'$.   The multiplicity $m_B$ and the height $h(B)$ are expressed in terms of 
invariants of the horn $H_B$ by the formulae of proposition \ref{conversion}.

Let $\gamma: x=\lambda (y)$ be a root of $A=A(B)$.  
Then the Puiseux characteristic exponents of $\gamma$ that are 
$< h(B)$ and the corresponding   signs of characteristic coefficients are those 
of $\gamma_{H_B,gen}$ (or, equivalently, of $\gamma_{H_B}$).  If $\tilde A= A (\tilde B)$ be a sub-bunch of
$A$ containing $\gamma$ then the invariants of $H_{\tilde B}$ determine 
 whether $\gamma$ 
takes a new Puiseux pair at $h(B)$ and, if this is the case, 
 the sign of the characteristic coefficient 
at $h(B)$.


\medskip
\subsection{End of proof of theorem \ref{c1givesblow}}
By propositions  \ref{lipinvariant2} and  \ref {conversion} 
the image of a root horn $H_B$ is a root horn $H_{\tilde B}$.  
Thus obtained one-to-one correspondence $B\leftrightarrow \tilde B$
gives an isomorphism of trees preserving the multiplicities and the heights 
of bars.  
The Puiseux characteristic exponents and the corresponding signs of
Puiseux coefficients are also preserved as follows from \ref{Puiseuxinvariance2}.  
If $\sigma$ preserves the orientation then it preserves the clockwise order of root horns and hence 
the clockwise order on the trees.    
Thus the theorem follows from theorem \ref{allequivalent}.  
\qed




\section {Arbitrary $C^1$ equivalence.}\label{arbitraryc1}

If $f$ and $g$ are $C^1-$equivalent by a  $C^1$ diffeomorphism 
$\sigma$, $f=g\circ \sigma$,  then usualy we 
 compose $f$ or $g$ with a linear isomorphism and assume that
 $D\sigma(0)=Id$.  
Nevertheless,  sometimes, it is necessary to construct 
invariants of the arbitrary $C^1-$equivalence.  This is the case when $f$ and $g$ are 
weighted homogeneous, a property that is usually destroyed by an
arbitrary linear change of 
coordinates.  In this section we construct invariants of the arbitrary
$C^1-$equivalence and apply them to weighted  homogeneous polynomials. 

\smallskip
\begin{prop}\label{c1general}
Let $\sigma : (\R^2,0) \to (\R^2,0)$ be a $C^1-$diffeomorphism  such that 
$D\sigma (0)  (x,y) = (ax+by,cx+dy)$ 
and let $f(x,y)$, $g(x,y)$ be real analytic function germs, mini-regular in $x$, 
such that $f = g \circ \sigma$.
Suppose that $\gamma$, $\tilde{\gamma}$ are allowable real analytic demi-branches and 
that there exist $ \xi_0  > 1, N > 0$ such that
$$
\sigma (\gamma ) \subset 
H_{\xi_0}(\tilde{\gamma};N) . 
$$
Then, for  $\xi \in (1,\xi_0)$, $P_{f,\gamma, \xi}$ and $P_{g, \tilde \gamma,\xi}$ are related by 
\begin{equation}\label{c1polynom}
P_{f,\gamma, \xi} (z) =  (c\lambda'(0) + d)^{\ord_\gamma f(\xi)} P_{g, \tilde \gamma,\xi} (\frac {ad-bc}{(c\lambda'(0) + d)^{\xi +1}} z) . 
\end{equation}
If $\xi =1$ then 
$$
P_{f,\gamma, 1} (z) =  (c\lambda'(0) + d + cz )^{m} P_{g, \tilde \gamma,1} 
(\frac {ad-bc}{c\lambda'(0) + d}  \cdot    \frac {z}{c\lambda'(0) + d
  + cz }). 
$$
\end{prop}
\medskip

\begin{example}\label{arnold}  Consider the family  
$$
A_t(x,y) = x^3 - 3 t x y^4 + 2 y^6.
$$
This family is equivalent to the family $J_{10}$ of \cite{AVG}.  
For each $t$, $A_t$ is mini-regular in $x$, and
$\frac {\partial A_t} {\partial x} = 3(x^2 - ty^4)$.
For $t>0$ let us consider the Newton polygon of $A_t$ relative to a polar curve
$\gamma_t : x = \sqrt t\,y^2$.
Then we have
$$
A_t(X+\sqrt t\,Y^2,Y) = X^3 + 3\sqrt t\,X^2Y^2 +
2(1 - t\sqrt t\,)Y^6 , 
$$
and 
$$
P_{\Gamma_t}(z) = z^3 + 3\sqrt t\,z^2 + 2(1 - t\sqrt t\,).
$$
Suppose that for $t$, $t^{\prime} \in (0,\infty)$,
there are $\alpha$, $\beta \ne 0$ such that
$P_{\Gamma_{t^{\prime}}}(z) = \beta^6 P_{\Gamma_t}
({\frac \alpha { \beta^3}}z)$.
By an easy computation, we obtain $\alpha^2 = \beta^2 = 1$ and that 
$P_{\Gamma_t} \equiv P_{\Gamma_{t^{\prime}}}$
up to a multiplication if and only if $t = t^{\prime}$ in this case.
\end{example}

\begin{proof}
By Proposition \ref{c1invariant2} it suffices to consider only the case of $\sigma $ linear   
\begin{equation*}
 (\tilde{x},\tilde{y})  = \sigma (x,y) = (ax+by,cx+dy),  \quad \det \sigma =ad-bc \ne 0 .  
\end{equation*}
and $\tilde \gamma = \sigma (\gamma)$.  Let $\gamma : x = \lambda (y)$,  $\tilde \gamma : x = \tilde  \lambda (y)$.  Then 
\begin{equation}
\lambda (y)a + by = \tilde{\lambda} (c\lambda (y)+dy),  
\end{equation} 
$c\lambda (y) + dy$ parametrises the positive 
$y$-axis,  and 
\begin{equation*}
\tilde{\lambda}^{\prime}(0) = \frac {a\lambda^{\prime}(0)+  b} {c \lambda^{\prime}(0)+d} ,  \qquad c\lambda^{\prime}(0) + d > 0 .
\end{equation*}

Fix $\xi >1$.  Clearly $\ord_\gamma f(\xi) =\ord_{\tilde \gamma} g(\xi)$.   
Put $Y = c(\lambda (y) + zy^\xi) + dy $.  
Then $Y = (c\lambda' (0) + d)y + o(y)$ and $y = (c\lambda' (0) + d)\inv Y + o(Y)$, and consequently   
\begin{eqnarray*}
& f (\lambda (y) + z y^\xi,y) & = g (a (\lambda (y) + z y^\xi) + by  , c(\lambda (y) + z y^\xi)+dy) \\
& &
= g ( \tilde \lambda (c\lambda (y)+ dy) +azy^\xi , Y) \\
& & 
= g (  \tilde \lambda (Y) - \tilde \lambda'(0) czy^\xi + azy^\xi  + o(y^\xi)  ,Y) \\
& & 
= g (  \tilde \lambda (Y)  +  \frac {ad-bc} { (c\lambda'(0) +d)^{\xi+1} } zY^\xi + o(Y^\xi)   ,Y) .  
\end{eqnarray*}
Hence, comparing this formula with \eqref{genericarc}, we get 
$$
P_{f,\gamma, \xi} (z) y ^{\ord_\gamma f(\xi)}  =  
P_{g, \tilde \gamma,\xi} (\frac {ad-bc}{(c\lambda'(0) + d)^{\xi +1}} z) Y^{\ord_{\tilde \gamma} g(\xi)} , 
$$
that gives \eqref{c1polynom}. 
The case $\xi=1$ is left to the reader.  
\end{proof}

\begin{cor}\label{coordsystem}
 Given an analytic function germ $f : (\R^2,0) \to (\R,0)$
and a real analytic demi-branch $\gamma$.
We say $(x,y)$ is an {\em admissible system of local analytic
  coordinates  for $f$ and $\gamma$} if $\gamma$ is allowable and 
 $f(x,y)$ is mini-regular in $x$.   
Then  $NB_{\gamma}f$ is independent of the choice of admissible
coordinate systems. 
Moreover, for each edge $\Gamma \subset NB_{\gamma}(f)$ with slope smaller than $-1$,
the polynomial $P_{\Gamma} (z) =  \sum_{(i,j) \in \Gamma} c_{ij} z^i    $  
is well-defined up to left and right multiplications  
as in \eqref{c1polynom}.
\end{cor}


\subsection{$C^1-$equivalent weighted homogeneous functions}  
Using Propositions \ref{c1general} and 
 \ref{lipinvariant2} we give below complete bi-Lipschitz and $C^1$ classifications of 
weighted homogeneous two variable function germs.  

Let $f(x,y)$ be a weighted homogeneous polynomial with weights 
$q,p$, $1\le p\le q$, $(p,q)=1$, and weighted degree $d$.  We may write    
\begin{equation}\label{formweighted}
f(x,y) = y^l (x^{d'/q} + \sum _{qi+pj=d'} a_{ij} x^i y^j) =  y^{d/p} P(x/y^\xi), 
\end{equation}
where $d'=d-pl$, and $\xi = q/p$. $P(z):= f(z,1)$ is the associated one variable polynomial. 
We distinguish the following three cases:
\begin{enumerate}[(A)] 
\item
homogeneous : $p=q=1$;
\item
$1=p<q$;
\item
$1<p<q$.
\end{enumerate}
In each of these cases we call the following polynomials \emph{monomial-like}: 
\begin{enumerate}[(Am)]
\item
$A(ax+by)^k (cx+dy)^l$,  $ad-bc\ne 0$;
\item
$A(x+by^q)^k y^l$;
\item
$A x^k y^l$.
\end{enumerate}

\begin{prop}\label{weighted}
Let $f(x,y)$ and $g(x,y)$ be weighted homogeneous polynomials and 
\begin{enumerate}
\item
suppose that $f$ and $g$ are bi-lipschitz equivalent. Then 
\begin{enumerate}
\item
If $f$ is monomial-like then so is $g$.  Then $f$ and $g$ are analytically equivalent. 
\item
if $f$ is not monomial-like then $f$ and $g$ have the same weights and weighted degree.  
\end{enumerate}
\item
suppose that $f$ and $g$ are $C^1$ equivalent and not monomial-like.  Fix the weights  
$q, p$, $1\le p\le q$, $(p,q)=1$.  Then 
\begin{enumerate}
\item
In case (A), $f$ and $g$ are linearly equivalent.  
\item
In case (B), there exist $c_1\ne 0,c_2 \ne 0$, and $b$ such that 
$$
f(x,y) = g(c_1x-by^q ,c_2y).
$$
\item
In case (C), there exist $c_1\ne 0,c_2 \ne 0$ such that 
$$
f(x,y) = g(c_1x,c_2y).
$$
\end{enumerate}
\end{enumerate}
\end{prop}

\begin{proof} 
Let $f$ be weighted homogeneous and let $\gamma$ be a demi-branch of a
root of $f$.  In this proof we shall call 
such $\gamma$ simply \emph{a root of $f$} for short.  
First we list all possiblities for 
the Newton boundary  $NB_{\gamma}f$ in an admissible system of coordinates, cf. corollary
\ref{coordsystem}.  
Note that in such a system of coordinates 
 $f$ may not be weighted homogeneous.  We denote  $m=\mult_0 f$ 
and by $m_\gamma$ the multiplicity of the root. 
\begin{enumerate}[(i)]
\item
If $f$ is monomial-like with $k=0$ or $l=0$ then  $m=m_\gamma$ and
$NB_{\gamma}f$ has only one vertex at $(m,0)$.
\item
If $f$ is monomial-like with $k\ne 0$ and $l\ne 0$, or homogeneous and not monomial-like,  
then $NB_{\gamma}f$ 
has two vertices at $(m,0)$, $(m_\gamma, m-m_\gamma)$, and hence one nontrivial compact edge of slope 
$-1$.   This is also the  Newton boundary for a not monomial-like non-homogeneous $f$ of the 
form \eqref{formweighted} and the root $y=0$.   
\item
If $f$ is not homogeneous and not monomial-like, of the 
form \eqref{formweighted} and $\gamma$ is not in $y=  0$, then we have two possiblities:
\begin{enumerate}
\item
If $l=0$ then $NB_{\gamma}f$ has one nontrivial edge of slope $-\xi$ and vertices $(m,0)$, 
$(m_\gamma, \xi (m-m_\gamma))$. 
\item
If $l\ne 0$ then $NB_{\gamma}f$ has two nontrivial edges: $\Gamma_1$ of slope $-1$ and vertices $(m,0)$, 
$(m-l,l)$, and $\Gamma_2$  of slope $-\xi$ and vertices 
$(m-l,l)$, $(m_\gamma , \xi (m-l-m_\gamma))$. 
\end{enumerate}
\end{enumerate}

\vspace{2mm}
\epsfxsize=14cm
\epsfysize=4cm
\begin{equation*} 
\epsfbox{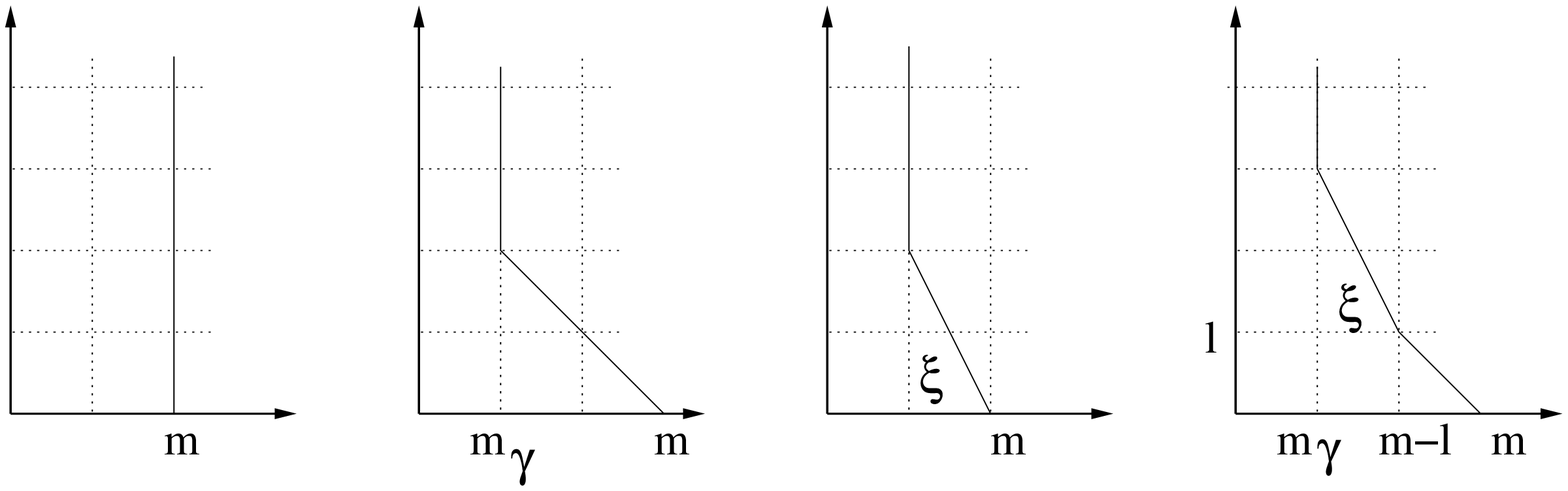}\end{equation*}
 
\vspace{2mm}

Let $f=g\circ \sigma$, $\sigma$ bi-Lipschitz.  Then $m=\mult_0 f
=\mult_0 g$ and $m_\gamma = m_{\tilde \gamma}$ if $\tilde \gamma =
\sigma (\gamma)$ for a root $\gamma$ of $f$.  
Moreover, $\sigma$ 
preserves the tangency of roots.  Therefore  $f$ is monomial if and only if it satisfies the following, 
bi-Lipschitz invariant, property:  $f$ has $s=2$ or $4$ roots (demi-branches), mutually not tangent, 
with the sum of multiplicities equal to $2 \mult_0 f$.   This shows (1a).  

If $f$ has a root $\gamma$ such that $NB_{\gamma}f$ contains an edge of slope $-\xi < -1$ then so does 
$g$, and $f$ and $g$ have the same weights.  Since the weighted degree can be also read from  
$NB_{\gamma}f$, they have the same weighted degree as well.  Thus to finish the proof of (1) it suffices 
to consider the following case.  

\smallskip
\emph{Special Case.}
Suppose that $f$ and $g$ are not monomial-like, and that for every root
 $\gamma$ and $\tilde \gamma$ of $f$ and $g$ respectively,
 $NB_{\gamma}f$ and $NB_{\tilde \gamma}g$ are of the form (ii).  
(This includes the case where both $f$ and $g$
 have isolated zero at the origin.)    In this case we shall replace the roots by horn neighbourhoods of 
 polar curves.  
 
 Suppose that the weights $q,p$ of $f$ satisfy $\xi = q/p>1$.  Write $f$ 
as in  \eqref{formweighted}.  
  Denote by $f_m$ and $g_m$ the 
leading homogeneous part of $f$ and $g$ respectively. 
The real analytic demi-branches $\delta$ tangent to a root of $f_m$ are distinguished  
by the size of $f$ on them, 
$f(x,y) = o(\| (x,y)\| ^m)$ for $(x,y)\in \delta$.
The positive (or similarly negative) $y$-axis is   
in the zero set of $f_m$ and is not tangent to any root of $f$.    Hence its image by 
$\sigma$ is in a horn 
neighbourhood of a root of $g_m$ that is not tangent to any root of $g$. Hence $g\ne g_m$,  
that is $g$ is not homogeneous.

By assumption, $P(z) = f(z,1)$ has no real root, and therefore $P'$ must have one.   If $P'(a)=0$ then the curve 
$\gamma_a : x= ay^\xi, y\ge 0,$ is a polar root of $f$ that is 
$$
\frac {\partial f} {\partial x}(ay^\xi, y)\equiv 0.
$$
Consider the germ at the origin of 
\begin{equation}
U_\varepsilon (f) = 
\{(x,y) \in \R^2 ;\, r \varepsilon \|\grad f (x, y)\| \le |f(x,y)| \} , 
\end{equation}
where $r=\|(x,y)\|$ and $\varepsilon >0$.   If $\varepsilon$ is sufficiently small then each polar root  
$\gamma_a $  is in  $U_\varepsilon (f)$.  Indeed, then  
$$
\|\grad f (ay^\xi, y)\|  = | \frac {\partial f} {\partial y}(ay^\xi, y)| = (d/p)  | P(a) y^{d-1} +\cdots | 
\simeq  r \inv (d/p) | f(ay^\xi, y)| .
$$ 
In general, if a real analytic demi-branch  
$$ \delta: x=\lambda (y) = a_\xi y^\xi + \sum_{i> N \xi} a_{i/N}y^{i/N} , \quad y\ge 0, 
 $$
is contained in $U_\varepsilon (f)$, then $P'(a_\xi)=0$ and  $a_{i/N}=0$ for  $\xi <i/N < 2\xi-1$.  
Hence $\delta$ is contained in a horn neighbourhood $H_{\mu}(\gamma_{a_\xi},M)$, with $\mu >\xi$.   
Consequently any local  (at the origin) connected component $U'$ of $U_\varepsilon (f) \setminus (0,0)$  
satisfies one of the following  properties: 
\begin{itemize}
\item
$U'$ is contained in a horn neighbourhood of a polar root $H_{\mu}(\gamma_a,M)$.  
 Then $f(x,y) \sim r^{d/p}$ on $U'$.  ($d/p>d/q=m$)
 \item
 $l>0$, c.f. \eqref{formweighted}, and $U'$ contains a real analytic demi-branch tangent to $y=0$ that is 
 a root of $f$.  
\item 
Otherwise $f(x,y) \sim r^m$ on any real analytic demi-branch in $U'$.  
\end{itemize}
By \cite{henryparusinski1} and \cite{henryparusinski2}, $\sigma (U_\varepsilon(f)) \subset 
U_{\varepsilon'}(g) $ and so the image of a local connected component of the first type has to be 
contained in a horn neighbourhood of a polar curve of $g$.   Thus the special case follows from
the following observation.   For any real analytic demi-branch $\delta$ in a horn neighbourhood 
$H_{\mu}(\gamma_a,M)$ of a polar curve  $\gamma_a$ of $f$, with $\mu >\xi$, the Newton boundary
 $NB_{\delta}f$ is independent of $\delta$ (we use $P(a)\ne 0$) and is of the form (iii).   
 This ends the proof of Special Case and completes the proof of (1).  
 \medskip

Now we show (2) of the proposition. (a) follows from lemma \ref{initialform}.  
Suppose $f$ is in the form \eqref{formweighted} with $\xi =q/p>1$.  We assume that 
$P$ has a root.  The proof in Special case is similar, one uses the polar roots instead of 
the roots.  
Let $P(a) =0$.  Then $\gamma : x= ay^\xi, y\ge 0,$ is a root of $f$ and 
 $\tilde \gamma = \sigma (\gamma)$ is a root of $g$.   
Replacing $g(x,y)$ by $g(-x,-y)$, if necessary, we 
may suppose that $\tilde \gamma : x=\tilde a y^\xi, y\ge 0$.     

 Let $\tilde P (z) = g(z,1)$.  Then $\tilde P(\tilde a)=0$.  Since $\sigma$ is $C^1$,  
by Proposition \ref{c1general},  $P_{f,\gamma,\xi}$ and $P_{g,\tilde \gamma,\xi}$ coincide up to 
the left and right multiplications. 
Multiplying $x$ by a positive constant, if necessary, we may suppose that 
\begin{equation}\label{polrel}
P(z-a) = \tilde P (\alpha(z-\tilde a)).
\end{equation}
For $p=1$ this gives $f(x,y) = g(c_1x-by^\xi ,c_2y)$ (taking into account of the changes 
we have made already) and ends the proof of (2b).

If $p>1$ then  
\begin{equation*}
P(z) = z^l Q(z^p),  \tilde P(z) = z^l \tilde Q(z^p).
\end{equation*}
and therefore the arithmetic mean of complex roots of $P$, and the one of the roots of $\tilde P$, 
equals $0$.   By \eqref{polrel},  if $z$ is a complex root of $P$ then 
$\alpha (z+a-\tilde a)$ is a root of $\tilde P$.  Thus by comparing both 
arithmetic means we get $a=\tilde a$. Consequently, 
$P(z-a) = \tilde P (\alpha(z-a))$ or,  
by replacing $z-a$ by $z$, $P(z) = \tilde P (\alpha z)$, and hence we may 
conclude finally that 
$$
f(x,y) = g(c_1x,c_2y). 
$$ 
This ends the proof of proposition \ref{weighted}.  
\end{proof}




\bigskip
\section{Bi-Lipschitz equivalence does not imply blow-analytic equivalence}
\label{bilipschitz}
\medskip

In this section we present several examples of bi-Lipschitz equivalent real analytic
function germs that are not blow-analytically equivalent. In order to
distinguish different blow-analytic types we use either the real tree
model of \cite{koikeparusinski2} or the Fukui invariants.  
Recall the definition of Fukui invariants of blow-analytic equivalence, c.f.  \cite{fukui}.
Let $f : (\R^n,0) \to (\R,0)$ be an analytic function germ.
Set
$$
A(f) := \{ \ord (f(\gamma (t))) \in \N \cup \{ \infty \} 
; \gamma : (\R,0) \to (\R^n,0) \ C^{\omega} \} .
$$ 
Let $\lambda : U \to \R^n$ be an analytic arc with $\lambda (0) = 0$, 
where $U$ denotes a neighbourhood of $0 \in \R$. 
We call $\lambda$  {\it nonnegative} (resp. {\it nonpositive}) 
{\it for} $f$ if 
$(f \circ \lambda)(t) \geq 0$ (resp. $\leq 0$) 
in a positive half neighbourhood $[0,\delta) \subset U$. 
Then we set

\vspace{3mm}

\qquad $A_+(f) := \ \{ \ord (f \circ \lambda) ; \lambda$
is a nonnegative arc through $0$ for $f \}$, 

\qquad $A_-(f) := \ \{ \ord (f \circ \lambda) ; \lambda$ 
is a nonpositive arc through $0$ for $f \}$.

\vspace{3mm}

\noindent Fukui proved that $A(f)$, $A_+(f)$ and $A_-(f)$ are
blow-analytic invariants.
Namely, if analytic functions $f, g : (\R^n,0) \to (\R,0)$
are blow-analytically equivalent, then $A(f) = A(g)$,
$A_+(f) = A_+(g)$ and $A_-(f) = A_-(g)$.
We call $A(f)$, $A_{\pm}(f)$ {\em the Fukui invariant},
{\em the Fukui invariants with sign}, respectively.
Apart from the Fukui invariants, motivic type invariants, 
{\em zeta functions}, are also known c.f. \cite{koikeparusinski}, \cite{fichou}.

\subsection{Example}\label{example1}  
$ \displaystyle {  f(x,y) = x(x^3-y^5), g(x,y) = x(x^3+y^5) }.$\\
By \cite{koikeparusinski2} $f$ and $g$ are not blow-analytically
equivalent by an orientation preserving blow-analytic homeomorphism.   

\vspace{2mm}
\epsfxsize=12cm
\epsfysize=2.2cm
$$\epsfbox{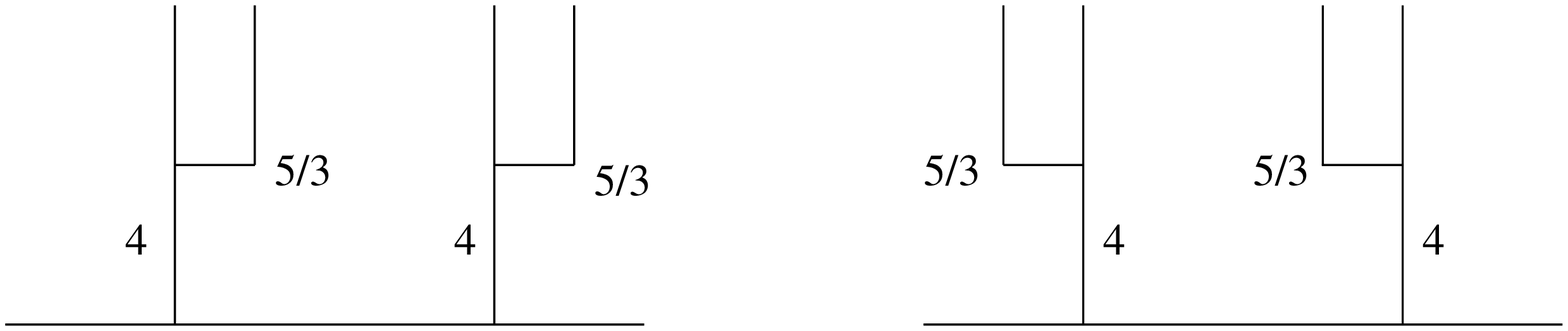}$$
\centerline{\hspace{10mm}$\R T(f)$\hspace{80mm} $\R T(g)$\hspace{10mm}}
\vspace{3mm}

We construct below an orientation preserving  bi-Lipschitz
homeomorphism 
$\sigma: (\R^2,0) \to (\R^2,0)$ such that $f=g\circ \sigma$. 
 The construction uses the fact that $f$ and $g$ are weighted homogeneous 
with weights $5$ and $3$.   Write 
\begin{eqnarray*} 
f(x,y) = x(x^3-y^5) = y^{20/3} P (\frac x {y^{5/3}}), \qquad P(z)  = z^4 -z, \\
g(x,y) = x(x^3+y^5)  =  y^{20/3} Q (\frac x {y^{5/3}}), \qquad Q(z) = z^4 +z.  
\end{eqnarray*}

\begin{prop}
There exists a unique increasing real analytic diffeomorphism $\varphi: \R \to \R$ 
satisfying $P=Q\circ \varphi$.  
Moreover, for this $\varphi$, $\varphi '$ and $\varphi - z \varphi '$ are globally bounded and 
$\varphi (z) /z\to 1$ as 
$z\to \infty$.  
\end{prop}

\begin{proof}
$P$ and $Q$ have unique critical points: $z_0 = \sqrt [3] {\frac 1 4}, P'(z_0) = 0 $, 
 $\tilde z_0 = -z_0, Q'(\tilde z_0) = 0 $.  
Therefore $\varphi: (-\infty, z_0] \to (-\infty, \tilde z_0]$, 
 defined as $Q\inv \circ P$, is continuous and analytic on  $(- \infty , z_0)$.  Similarly for 
 $\varphi: [z_0, \infty) \to [\tilde z_0, \infty)$.  Thus $\varphi : \R \to \R$ is well-defined and continuous.   
  In a neighbourhood of $z_0$, that is a non-degenerate critical
  point, $P$ is analytically equivalent to $- z^2 + P(z_0) $.  
Similarly $Q$ near $\tilde z_0$ is analytically equivalent to 
 $- z^2 + Q(\tilde z_0) $.   Finally, since $P(z_0)=Q(\tilde z_0)$,  
$P$ near  $z_0$ is analytically equivalent to $Q$ near 
 $\tilde z_0$.   
 
 Let $w= \frac 1 z$.  Consider real analytic function germs 
 \begin{eqnarray*} 
p(w) : = (P (w\inv  ))\inv : (\R,0)\to (\R,0), \quad p(w) = w^4  + \cdots , \\
q(w) : = (Q (w\inv  ))\inv : (\R,0)\to (\R,0), \quad q(w) = w^4  + \cdots . 
\end{eqnarray*}
Then $p=q\circ \psi$ with $\psi (w) = w + \cdots$.  Since $\varphi(z) = (\psi(z \inv ))\inv $, 
the last claim of proposition can be verified easily.  
\end{proof}

\begin{cor}
$\sigma :(\R^2,0) \to (\R^2,0)$, defined by 
\begin{equation*}
\sigma (x,y) = 
\begin{cases} (y^{5/3} \varphi  (\frac x {y^{5/3}}),y) & \text{if $y\ne 0$,}
\\
(x,0) &\text{if $y= 0$, }
\end{cases}
\end{equation*}
is bi-Lipschitz and $f=g\circ \sigma$.  
\end{cor}

\begin{proof}
We only check that $\sigma$ is Lipschitz. 
 This follows from the fact that the partial derivatives of 
$\sigma$ are bounded 
\begin{equation*} 
\partial \sigma /\partial x = (\varphi'( z), 0) , \quad 
\partial \sigma /\partial y = (5/ 3 y ^{2/3} (\varphi (z) - z \varphi'(z)), 1) . 
\end{equation*}
where $z=  \frac x {y^{5/3}}$.  
\end{proof}

\subsection{Example}\label{example2}  
$ \displaystyle {  f(x,y) = x(x^3- y^5)(x^3+y^5) , g(x,y) =  x(x^3-ay^5) (x^3-by^5) },$\\
where $0<a<b$ are constants.  The real trees of $f$ and $g$ are not
equivalent, see below, hence by 
 \cite{koikeparusinski2}, $f$ and $g$ are not blow-analytically
 equivalent.

\vspace{3mm}
\epsfxsize=10cm
\epsfysize=3cm
$$\epsfbox{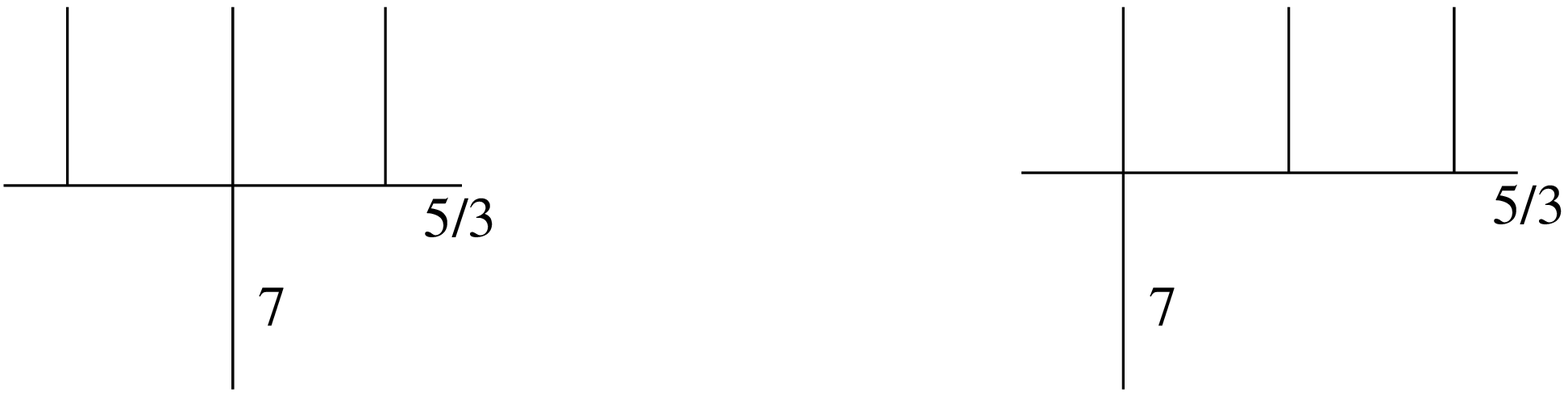}$$
\centerline{\hspace{10mm}$\R T_{(0,1)}(f)$\hspace{80mm} $\R T_{(0,1)}(g)$\hspace{10mm}}
\vspace{2mm}

Note that the Fukui invariants and the zeta functions of $f$ and $g$ coincide
(cf. Example 1.4 in \cite{koikeparusinski2}).
We show below that for a choice of $a$ and $b$, $f$ and $g$ are 
bi-Lipschitz equivalent.   Write 
\begin{eqnarray*} 
f(x,y) = y^{35/3} P (\frac x {y^{5/3}}), \qquad P(z) = z(z^3 -1)(z^3+1), \\
g(x,y) =   y^{35/3} Q (\frac x {y^{5/3}}), \qquad Q (z)  = z(z^3 -a)(z^3-b) .  
\end{eqnarray*}

The polynomial $P$ has two non-degenerate critical points  $-1< z_1 <0 $, $ z_2
 = -  z_1$ and $P( z_1)>0, P(z_2)= - P( z_1)<0$.  
 The polynomial $Q$ has also  two non-degenerate critical points  
$0< \tilde z_1 < \sqrt[3] a < \tilde z_2 <  \sqrt[3] b$ and $Q( \tilde z_1)>0, 
Q(\tilde z_2) < 0$.  Indeed, the discriminant of 
  $Q'(z) = 7z^6 -4(a+b) z^3 +ab$ with respect to $z^3$ equals $\Delta
  = 4(4a^2 +4b^2 +ab)>0$.  
This also shows that these critical points 
  $\tilde z_1 (a,b), \tilde z_2 (a,b)$ depend smoothly on $a,b$.  
  
\begin{lem}\label{ab}
There exist $a,b$, $0<a<b$, such that $Q(\tilde z_1(a,b)) = P(z_1)$,  $Q(\tilde z_2(a,b)) =P(z_2)$.  
\end{lem}
  
\begin{proof}
Fix  $b>0$.   If $a\to 0$ then $Q(\tilde z_1) \to 0$ and  $Q(\tilde z_2) \to const <0$.  
If $a\to b$ then $Q(\tilde z_1) \to const >0$ and  $Q(\tilde z_2) \to 0$.  Therefore there is an $a(b)$ such that 
$Q(\tilde z_1(a(b),b)) = - Q(\tilde z_2(a(b),b))$.   

 Write $Q_{a,b}$ instead of $Q$ to emphasise that $Q$ depends on $a$ and $b$.  
If $\alpha>0$ then $Q_{a,b}(\alpha z) = \alpha^7
Q_{a/\alpha^3,b/\alpha^3}(z)$.  Thus, there is $\alpha >0$ such that  the critical values of 
$ Q_{a(b)/\alpha^3,b/\alpha^3}$ are precisely $P( z_1), P(z_2)$.  This shows the lemma.  
\end{proof}   

Then, for $a$ and $b$ satisfying lemma \ref{ab},  the construction of bi-Lipschitz homeomorphism $\sigma$ such 
that $f=g\circ \sigma$ is similar to that of example \ref{example1}.


\medskip
\subsection{Example}\label{example3}
\begin{eqnarray}\label{lastexample}
& & f(x,y) =  x(x^3 - y^5)((x^3 - y^5)^3 - y^{17}) \\ 
\notag
& &  g(x,y) = x(x^3 + a y^5)(x^3 - y^7)(x^6 + by^{10}) , 
\end{eqnarray}
where $a>0, b>0$ are real constants.      
As we show below, for a choice of $a$ and $b$, $f$ and $g$ are 
bi-Lipschitz equivalent.  They have different real tree models, see
below, so they are not blow-analytically equivalent.  
Moreover, in contrast to the previous two examples, $f$ and $g$ have different Fukui invariants.  

\vspace{3mm}
\epsfxsize=10cm
\epsfysize=3cm
$$\epsfbox{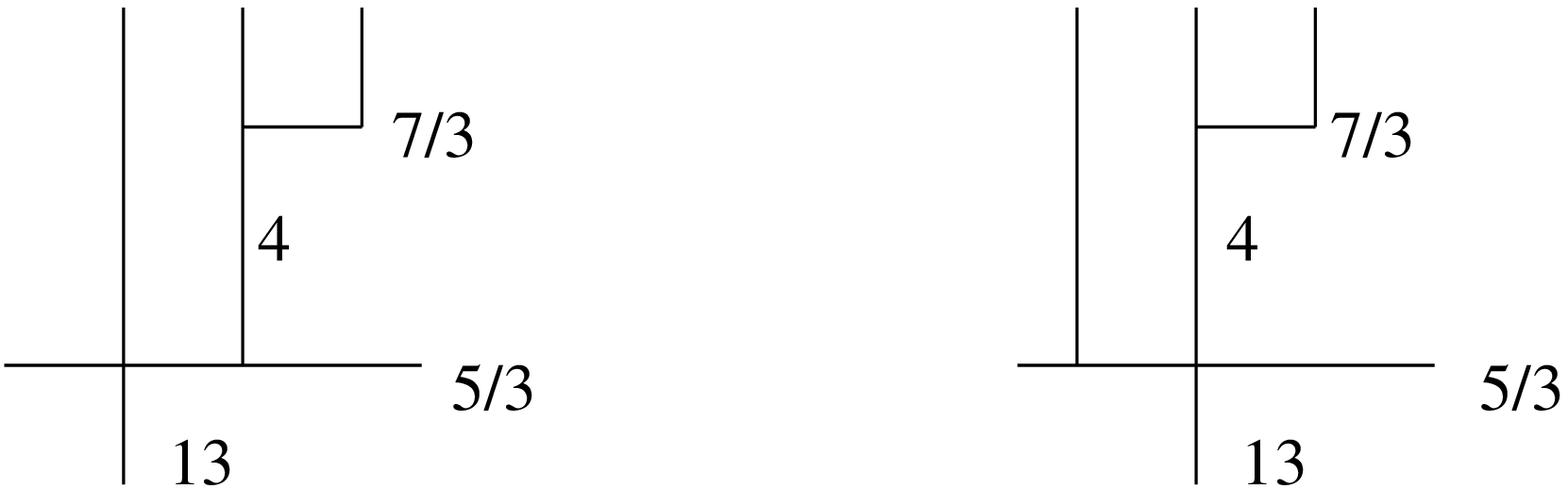}$$
\centerline{\hspace{10mm}$\R T_{(0,1)}(f)$\hspace{80mm} $\R T_{(0,1)}(g)$\hspace{10mm}}
\vspace{2mm}

\begin{prop}\label{differentinv}
Let $f, \ g : (\R^2,0) \to (\R,0)$ be polynomial functions  defined by
\eqref{lastexample}.   
Then 
$$
A(f) = \{ 13, 22, 23, 24, \cdots \} \cup \{ \infty \} , \ \ \
A(g) = \{ 13, 23, 25, 26, \cdots \} \cup \{ \infty \} .
$$
Thus $f$ and $g$ are not blow-analytically equivalent.
\end{prop}

\begin{proof}
Let us express an analytic arc at $(0,0) \in \R^2$,
$\lambda (t) = (\lambda_1(t),\lambda_2(t))$, as follows:
$$
\lambda_1(t) = c_1t + c_2t^2 + \cdots , \ \ \
\lambda_2(t) = d_1t + d_2t^2 + \cdots .
$$
To compute $A(f)$, we consider $f(\lambda (t))$;
$$
f(\lambda (t)) = 
(c_1t + c_2t^2 + \cdots )(c_1^3t^3 + \cdots - d_1^5t^5 - \cdots )
((c_1^3t^3 + \cdots - d_1^5t^5 - \cdots )^3 - d_1^{17}t^{17} - \cdots ).
$$
In case $c_1 \ne 0$, we have $\ord (f \circ \lambda ) = 13$.
In case $c_1 = 0$, we have $\ord (f \circ \lambda ) \ge 22$.
For any $s = 2, 3, \cdots$, $20 + s$ is attained by the arc
$\lambda (t) = (t^s,t)$.
Therefore we have
$$
A(f) = \{ 13, 22, 23, 24, \cdots \} \cup \{ \infty \} .
$$

We next compute $A(g)$. Then

\vspace{3mm}

\centerline{$g(\lambda (t)) = (c_1t + c_2t^2 + \cdots )
(c_1^3t^3 + \cdots + a d_1^5t^5 + \cdots )(c_1t^3 + \cdots - d_1^7t^7 - \cdots )$}

\vspace{3mm}

\centerline{$\times
(c_1^6t^6 + \cdots + bd_1^{10}t^{10} + \cdots ).$}

\noindent In case $c_1 \ne 0$, we have $\ord (f \circ \lambda ) = 13$.
In case $c_1 = 0$, $c_2 \ne 0$ and $d_1 \ne 0$,
we have $\ord (f \circ \lambda ) = 23$.
In case $c_1 = c_2 = 0$ or $c_1 = d_1 = 0$,
we have $\ord (g \circ \lambda ) \ge 25$.
For any $s = 3, 4, \cdots$, $22 + s$ is attained by the arc
$\lambda (t) = (t^s,t)$.
Therefore we have
$$
A(g) = \{ 13, 23, 25, 26, \cdots \} \cup \{ \infty \} .
$$
\end{proof}

Next we compute the polar roots of $f$ and $g$.  The one variable polynomial associated to the 
leading weighted homogeneous part of $f$ with respect to the weigths
$5$ and $3$ equals  $P_1 (z)= z(z^3-1)^4$.  
Besides a multiple root $z=1$, it 
has a unique non-denegenerate critical point $a_1$, $0<a_1<1$, which gives rise to a polar curve 
$$
\gamma_1 : x=\lambda_1(y) = a_1 y^{5/3} + \cdots ,  \qquad f(\lambda_1(y),y) = 
A_1 y^{21\frac 23} + O( y^{23\frac 23} )
$$
where $A_1= P_1(a_1)$. The Newton polygon of $f$ relative to 
$\gamma : x=y^{5/3}$ has two edges: one of slope $-5/3$ and one of slope $-7/3$.  
The one variable polynomial associated to the latter is 
$P_2 (z) := P_{f,\gamma,7/3}(z) = 3^4z^4 -3z$.  
The unique 
non-degenerate critical point $a_2$ of $P_2$ gives rise to a polar curve 
$$
\gamma_2 : x=\lambda_2(y) =  y^{5/3} + a_2  y^{7/3} +\cdots ,  \qquad f(\lambda_1(y),y) = A_2 y^{24\frac 13} + 
O( y^{25} )
$$
where $A_2= P_2(a_2)$.  

The one variable polynomial associated to the 
leading weighted homogeneous part of $g$ equals  $Q_1(z) = z^4 (z^3+a) (z^6 +b)$.  
If  $ 10^2 a^2 - (7\cdot 39)b <0 $ then  
 $Q_1 '(z) = 13 z^{12} + 10 a z^{9} + 7 b z^6 + 4ab  z^3$ has a single simple non-zero real root   
  Indeed, let $S(t) = 13 t^{3} + 10 a t^{2} + 7 b t + 4ab $.  
 Then $S' (t) = 39 t^{2} + 20 a t^{} + 7 b $ and the discriminant of $S'(t)$ is $\Delta /4 = 10^2 a^2 - (7\cdot 39)b$.  
 Therefore, if we suppose that  
\begin{equation}\label{abnew}
a>0, b>0, 10^2 a^2 - (7\cdot 39)b <0, 
\end{equation}  
then $S(t)$ has a single simple root, that shows our claim on $Q_1'$.  
Let $\tilde a_1$ denote this 
critical point of  $Q_1$, $\tilde a_1<0$.  Then there exists a polar curve of $g$   
\begin{eqnarray*}
 \tilde \gamma_1 : x= \tilde \lambda_1(y) =  \tilde a_1 y^{5/3} + \cdots ,  \qquad g(\tilde \lambda_1(y),y) = 
\tilde A_1 y^{21\frac 23} + O( y^{23\frac 23} ). 
\end{eqnarray*}
Finally, the one variable polynomial associated to the face of the Newton polygon of $g$ of slope $-7/3$ is 
$Q_2 (z) = z^4 -z$.  It has a single non-degenarate critical point  $\tilde a_2$ that gives a polar curve 
\begin{eqnarray*}
\tilde \gamma_2 : x= \tilde \lambda_2(y) =  \tilde a_2  y^{7/3} +\cdots ,  \qquad g(\tilde \lambda_1(y),y) = 
\tilde A_2 y^{24\frac 13} + O( y^{26\frac 1 3} )
\end{eqnarray*}
where $\tilde A_2=Q_2(\tilde a_2)$.   One checks easily that $\tilde A_2=A_2$.  

\begin{lem}\label{abA}
There are constants $a, b$ satisfying \eqref{abnew} for which $\tilde A_1=A_1$.  
\end{lem}

\begin{proof}
Denote by  $\tilde a_1(a,b)$ the unique non-zero critical point of $Q_1$ thus emphasising that it depends 
on $a,b$.    Note that $\tilde a_1(a,b)$ is between the two roots of $Q_1$,  
$-\sqrt [3] a <\tilde a_1 (a,b) <0$.  
For $b$ fixed $ Q_1( \tilde a_1(a,b)) \to 0$ as $a\to 0$.  
Fix $a$ and let $b \to \infty$.  
Then $Q_1( - \half \sqrt [3] a ) \to \infty$ and hence $Q_1( \tilde a_1(a,b))  
\to \infty$.  
Thus there exist $a,b$ for which $Q_1(\tilde a_1(a,b)) = A_1$.  
\end{proof}


Next for $f$, and then for $g$, we introduce a new system of local coordinates 
$(\tilde x, \tilde y) = H(x,y)$ 
in which $f$ has particularly simple form near the polar curves.  
Firstly, for each polar curve $\gamma_i$, $i=1,2,$ separately, we
reparametrise $\lambda_i$ 
by replacing $y$ by an invertible fractional power series 
$\tilde y_i(y)$ 
so that 
\begin{eqnarray} \label{reparameter}
& &  f(\lambda_1 (y(\tilde y_1)), y(\tilde y_1)) = 
A_1 \tilde y_1^{21\frac 23} ,  \quad \tilde y_1 = y+ O(y^{3}), \\ 
& &  f(\lambda_2 (y(\tilde y_2)), y(\tilde y_2)) = 
A_2 \tilde y_2^{24\frac 13} ,  \quad \tilde y_2 = y+ O(y^{5/3}).
\end{eqnarray}

Denote $\xi =5/3$ for short.  
Let $\varphi_0, \varphi_1, \varphi_2$ be a ($C^\infty$ or $C^k, k\ge
2$, semialgebraic) partition of unity on $\R$ such that 
\begin{enumerate} [(i)]
\item 
$\supp \varphi_1$ is a small neighbourhood of $a_1$ and $\varphi_1
\equiv 1$ in a neighbourhood of $a_1$.
\item 
$\supp \varphi_2$ is a small neighbourhood of $1$ and $\varphi_2
\equiv 1$ in a neighbourhood of $1$.
\end{enumerate}
Then $\varphi_0 = 1 - \varphi_1 - \varphi_2$.  We set $\tilde y_0 (y) =y$ and define 
\begin{eqnarray*}
& & \Phi(x,y) = (x, \tilde y(x,y)) = (x, \sum_{i=0}^2 \tilde y_i(y)
\varphi_i (x/y^\xi) )  , \quad   \Phi(x,0) = (x,0).
\end{eqnarray*}

Let $\psi :\R \to \R$ be a ($C^\infty$ or $C^k, k\ge 2$, semialgebraic) diffeomorphism such that 
\begin{enumerate} [(i)]
\item 
$\psi (a_1)=0$ and $\psi (z) = z-a_1$ for $z$ near $a_1$. 
\item 
$\psi (1) = 1$ and $\psi (z) = z$ for $z$ near $1$.  
\item
$\psi (z) = z$ for $|z|$ large.
\end{enumerate}
We set 
\begin{eqnarray*}
& & \Psi_1(x,  y) = (\psi (x/  y^\xi)  y^\xi, y)  , \quad   \Psi_1(x,0) = (x,0).
\end{eqnarray*}

Let $\psi_0, \psi_1, \psi_2$ be a ($C^\infty$ or $C^k, k\ge 2$, semialgebraic) partition of unity on $\R$ 
such that 
\begin{enumerate} [(i)]
\item 
$\supp \psi_1$ is a small neighbourhood of $0$ and $\psi_1 \equiv 1$ in a neighbourhood of $0$.
\item 
$\supp \psi_2$ is a small neighbourhood of $1$ and $\psi_2 \equiv 1$ in a neighbourhood of $1$.
\end{enumerate}
Let $x=\delta_1 (y)$ be an equation of $\Psi_1\circ \Phi (\gamma_1)$ and 
let $x=\delta_2 (y) + y^\xi$ be an equation of $\Psi_1\circ \Phi (\gamma_2)$.  
Note that $\delta_i(y) = o(y^\xi)$, $i=1,2$.   
 We set $\delta_0 \equiv 0$ and define 
\begin{eqnarray*}
& & \Psi_2 (x,y) = ( \sum_{i=0}^2 (x- \delta_i(y)) \psi_i (x/y^\xi)
,y), \quad   \Psi_2(x,0) = (x,0).
\end{eqnarray*}

\begin{prop}\label{change}
$H = \Psi_2 \circ \Psi_1 \circ \Phi$ and 
$\tilde f (\tilde x,\tilde y) = f \circ H\inv (\tilde x,\tilde y)$  satisfy 
the following properties: 
\begin{enumerate}
\item
$H$ is a bi-Lipschitz local homeomorphism.  Moreover, $D^2 H= O(y^{-\xi})$.
\item
$\{ \partial \tilde f /\partial \tilde x = 0 \} = H( \{ \partial  f
/\partial  x = 0 \}) = H(\gamma_1)\cup H(\gamma _2)$
and $H(\gamma_1) = \{\tilde x =0\}$, $H(\gamma_2) = \{\tilde x =\tilde y^\xi \}$. 
\item 
In a horn neighbourhood of $\gamma_1$, $\gamma_2$ resp.,  with exponent $\xi$, $H$ is given by 
$$
H(x,y) = (x - \lambda_1 (y), \tilde y_1(y)) , \qquad H(x,y) = (x - \lambda_2 (y) + \tilde y_2^\xi (y), \tilde y_2(y))
$$
\item
For $C$ large and $|x| \ge C |y|^\xi$, $H(x,y) = (x,y)$.  
\end{enumerate}
\end{prop}

\begin{proof}
(3) and (4) are given by construction.

We show that the partial derivatives of $\Phi$, $\Psi_1$, and $\Psi_2$ are bounded.  
For $\Phi$ it is convenient to write 
$\Phi(x,y) = (x , y  +  \sum_{i} ( \tilde y_i - y)  \varphi_i (x/y^\xi) )$.  Then 
\begin{eqnarray*}
& & \partial \Phi/\partial x = (1, \sum_i (\tilde y_i - y) y^{-\xi}  \varphi_i')  \qquad \text { bounded},\\
& & \partial \Phi/\partial y = (0, 1- \sum_i \xi \frac x {y^\xi} (\tilde y_i - y) y^{-1} \varphi_i'  +  
\sum_{i} ( \tilde y'_i - 1)  \varphi_i ) = (0,1 +o(y)) , \\
& & \partial \Psi_1/\partial x = (\psi',  0) ,\\
& & \partial \Psi_1/\partial y = (\xi y^{\xi-1}(\psi - \frac x {y^\xi}
\psi'  )  , 1) = (o(y), 1).   
\end{eqnarray*}
For $\Psi_2$ it is convenient to write 
$\Psi_2 (x,y) = (x-  \sum_{i} \delta_i(y)  \psi_i (x/y^\xi),y)$
\begin{eqnarray*}
& & \partial \Psi_2/\partial x = (1 - \sum_i \delta_i(y) y^{-\xi} \psi',  0) = (1-o(y), 0) ,\\
& & \partial \Psi_2/\partial y = (- \sum_{i} \delta'_i(y)  \psi_i  +
\sum_{i} \xi \frac x {y^\xi}  y^{-1}  \delta_i(y)  
\psi' , 1) = (o(y), 1).   
\end{eqnarray*}
Thus $H$ is Lipschitz, $H\inv (0) = 0$, and $H$ is a covering over the complement of the origin.  
Hence it is invertible.  
The formulae for the partial derivatives 
also show that the the inverse of Jacobian matrix of $H$ has bounded 
entries.  Thus $H\inv$ is also Lipschitz.  The last formula of (1) can be verified directly.

To show (2) we note that 
$$
\frac {\partial \tilde f } {\partial \tilde x} = \frac {\partial f} {\partial x}  \cdot  
 \frac {\partial  x} {\partial \tilde x}  + \frac {\partial f} {\partial y} 
 \cdot  \frac {\partial  y} {\partial \tilde x}    .  
$$ 
Note that there is a constant $c>0$ such that 
$c\le  \frac {\partial  x} {\partial \tilde x}\le c\inv$.  
(2) can be verified easily in the 
horn neighbourhoods considered in (3) and for $|x/y^\xi|$ large by
(4).  In the complement of these  sets 
$\partial f/\partial x \sim y^{20}$ and $\partial f/\partial y = O(y^{21\frac 23 })$ and hence 
\begin{equation*}
\frac {\partial \tilde f } {\partial \tilde x} \sim  y^{20} \sim  \tilde y^{20},
\end{equation*}
and does not vanish.  This shows  
$\{ \partial (f \circ \Phi \inv) /\partial x = 0 \} 
= \Phi ( \{ \partial  f /\partial  x = 0 \})$.  Similar results for 
$\Psi_1$ and $\Psi_2$ are obvious.  
\end{proof}

We apply the same procedure to $g(x,y)$ and obtain a bi-Lipschitz homeomorphism $\tilde H$ so that 
$\tilde H$ and $\tilde g(\tilde x, \tilde y)$ satisfy the statement of
Proposition \ref{change}.  In what follows 
we shall drop the \textquotedblleft tilda" notation for variables and consider $\tilde f$ and $\tilde g$ as functions of $(x,y)$.  
 We show 
that the homotopy 
$$
F( x,  y, t) = t \tilde g(x,y) + (1-t) \tilde f (x,y) 
$$
is bi-Lipschitz trivial and can be trivialised by the vector field 
\begin{equation}\label{vectorfield}
v(x,y,t) = \frac \partial {\partial t}  - \frac {\partial F/\partial
  t}  {\partial F/\partial x}   
\frac \partial {\partial x}, \qquad 
v(x,y,t) = \frac \partial {\partial t}   \text { if }  \partial F/\partial x=0 .
\end{equation}
Thus to complete the proof of bi-Lipschitz equivalence of $f$ and $g$
it suffices to show:   

\begin{lem}
The vector field $v(x,y,t)$ of \eqref{vectorfield} is Lipschitz.  
\end{lem}

\begin{proof}
The polar curves of $\tilde f$ and $\tilde g$ coincide :  
\begin{equation}\label{bothpolars}
\{ \partial \tilde f /\partial \tilde x = 0 \} 
=  \{ \partial \tilde g /\partial  x = 0 \} = \{x=0\} \cup \{x=y^\xi\}.
\end{equation}
As we shall show also 
$
\{ \partial \tilde F  /\partial \tilde x = 0 \} =  \{x=0\} \cup 
\{x=y^\xi\}$.  

We proceed separately in each of the horn neighbourhood with exponent $\xi$ 
of the polar curves \eqref{bothpolars}, for $|x|\ge C|y|^\xi$, $C$
large, and in the complement of these three sets.  

Suppose $|x|\le \varepsilon |y|^\xi$, $\varepsilon >0$ and small.  
By (iii) of Proposition \ref{change}, $\tilde f$ and $\tilde g$ are fractional 
convergent power series in $x$ and $y$.  
If we pass to new variables $z= x/y^\xi, y$ then, 
thanks to \eqref{reparameter}, 
\begin{eqnarray*}
\partial F/\partial x 
= z y^{20} u(z,y,t) \qquad 
\frac {\partial F} {\partial t} 
= \tilde g - \tilde f = z^2 y^{21 \frac 23} \eta(z,y), 
\end{eqnarray*}   
where $u$ and $\eta$ are fractional power series and $u(0)\ne 0$. Hence 
$$
 \frac {\partial F/\partial t}  {\partial F/\partial x} = z y^\xi h(z,y,t) = 
x h(z,y,t).
$$
Thus $\frac {\partial F/\partial t}  {\partial F/\partial x} $ is Lipschitz 
because the partial derivatives of $x h(z,y,t)$ 
are bounded: 
\begin{eqnarray*}
\frac {\partial}  {\partial x} (x h) = h + \frac x {y^\xi} \frac  
{\partial h} {\partial x} , \, \, 
\frac {\partial}  {\partial y} (x h) = 
x \frac {\partial h} {\partial z} \frac  
{\partial z} {\partial y} + x \frac {\partial h} {\partial y} = 
- \xi \frac x {y^\xi}\frac x {y} \frac  
{\partial h} {\partial z} + x \frac {\partial h} {\partial y}, \, \, 
\frac {\partial}  {\partial t} (x h) =  x  \frac  
{\partial h} {\partial t} . 
\end{eqnarray*}
A similar argument works for a horn neighbourhood of $x=y^\xi$.  

If $|x|\ge C|y|^\xi$, $C$ large, then by (iv) of Proposition \ref{change}, 
$\tilde f =f$ and $\tilde g = g$.  Then in  variables $x, w= y^\xi/x$ 
\begin{eqnarray*}
\partial F/\partial x 
= x^{m-1} u(x,w,t) \qquad 
\frac {\partial F} {\partial t} 
=  x^m \eta(x,w), 
\end{eqnarray*}   
where $u$ and $\eta$ are fractional power series and $u(0)\ne 0$.
Hence 
$$
 \frac {\partial F/\partial t}  {\partial F/\partial x} = x h(x,w,t) .
$$
Then, an elementary computation shows that the partial derivatives of $x h(x,w,t)$ 
are bounded.  

Suppose now that $x/y^\xi$ is bounded and that we are not in horn neighbourhoods of the polar curves.  
By Proposition \ref{change} one can 
 verify easily that on this 
set 
\begin{eqnarray*}
& & \tilde g - \tilde f = O( y^{20+\xi}),  \quad D(\tilde g - \tilde f) = O( y^{20+\xi}), \\
& & D^2 H\inv = O( y^{-\xi}) \\
& &  \partial F/\partial  x \sim  y^{20} ,  \quad D (\partial F/\partial  x) = O(  y^{20-\xi} ).
\end{eqnarray*}
Now a direct computation shows that the partial derivatives of 
$ \frac {\partial F/\partial t}  {\partial F/\partial x}$ are bounded.   
\end{proof}



\end{document}